\documentclass[11pt]{amsart}
\usepackage{amssymb,amsthm,amsmath, ulem, fullpage}
\usepackage{alltt}

\input xy
\xyoption{all}

\DeclareFontFamily{OMS}{rsfs}{\skewchar\font'60}
\DeclareFontShape{OMS}{rsfs}{m}{n}{<-5>rsfs5 <5-7>rsfs7 <7->rsfs10 }{}
\DeclareSymbolFont{rsfs}{OMS}{rsfs}{m}{n}
\DeclareSymbolFontAlphabet{\scr}{rsfs}

\newtheorem{theorem}{Theorem}[section]
\newtheorem{lemma}[theorem]{Lemma}
\newtheorem{proposition}[theorem]{Proposition}
\newtheorem{corollary}[theorem]{Corollary}

\newtheorem*{mainthm}{Main Theorem}
\newtheorem*{atheorem}{Theorem A}
\newtheorem*{ntheorem}{Theorem B}
\newtheorem*{ctheorem}{Theorem C}

\theoremstyle{definition}
\newtheorem{definition}[theorem]{Definition}
\newtheorem{example}[theorem]{Example}

\newtheorem{observation}[theorem]{Observation}

\theoremstyle{remark}
\newtheorem{remark}[theorem]{Remark}

\newtheorem{question}[theorem]{Question}

\newcommand{\sH}{\scr{H}}

\DeclareMathOperator{\Ann}{{Ann}}
\DeclareMathOperator{\divisor}{{div}}
\DeclareMathOperator{\Div}{{div}}

\DeclareMathOperator{\Hom}{Hom}
\DeclareMathOperator{\sHom}{{\sH}om}
\DeclareMathOperator{\Spec}{{Spec}}

\DeclareMathOperator{\Ext}{Ext}

\newcommand{\tld}{\widetilde }
\newcommand{\ba}{\mathfrak{a}}
\newcommand{\bm}{\mathfrak{m}}
\newcommand{\blank}{\underline{\hskip 10pt}}
\newcommand{\bF}{\mathbb{F}}
\newcommand{\bQ}{\mathbb{Q}}
\newcommand{\bZ}{\mathbb{Z}}
\newcommand{\sL}{\scr{L}}
\newcommand{\mydot}{{{\,\begin{picture}(1,1)(-1,-2)\circle*{2}\end{picture}\ }}}
\newcommand{\tensor}{\otimes}

\renewcommand{\O}{\mbox{$\mathcal{O}$}}

\begin{document}

\title{$F$-adjunction}
\author{Karl Schwede}

\thanks{The author was partially supported by a National Science Foundation postdoctoral fellowship.}
\address{Department of Mathematics\\ University of Michigan\\ East Hall
530 Church Street \\ Ann Arbor, Michigan, 48109}
\email{kschwede@umich.edu}
\subjclass[2000]{14B05, 13A35}
\keywords{F-pure, F-split, test ideal, log canonical, center of log canonicity, subadjunction, adjunction conjecture, different}
\begin{abstract}
In this paper we study singularities defined by the action of Frobenius in characteristic $p > 0$.  We prove results analogous to inversion of adjunction along a center of log canonicity.   For example, we show that if $X$ is a Gorenstein normal variety then to every normal center of sharp $F$-purity $W \subseteq X$ such that $X$ is $F$-pure at the generic point of $W$, there exists a canonically defined $\bQ$-divisor $\Delta_{W}$ on $W$ satisfying $(K_X)|_W \sim_{\bQ} K_{W} + \Delta_{W}$.  Furthermore, the singularities of $X$ near $W$ are ``the same'' as the singularities of $(W, \Delta_{W})$.  As an application, we show that there are finitely many subschemes of a quasi-projective variety that are compatibly split by a given Frobenius splitting.  We also reinterpret Fedder's criterion in this context, which has some surprising implications.
\end{abstract}

\maketitle

\section{Introduction}

Suppose that $X$ is a variety and $Y$ is an effective integral Weil divisor on $X$ such that $n(K_X + Y)$ is Cartier.  If the singularities of $X$ are mild (for example, if $X$ is Cohen-Macaulay and normal) one has a restriction theorem $\omega_X(Y) / \omega_X = \omega_Y$.  However $\O_X(n(K_X + Y))|_Y$ is not necessarily equal to $nK_Y$; there is an additional residue of $\O_X(n(K_X + Y))|_Y$ which (when divided by $n$) is called ``the different'', see \cite[Lemma 5-1-9]{KawamataMatsudeMatsuki} and \cite[Chapter 16]{KollarFlipsAndAbundance}.  Even when $Y$ is an arbitrary subvariety (that is, not a divisor) similar phenomena have been observed, see for example  \cite{KawamataSubadjunctionOne}, \cite{KawamataSubadjunction2}, \cite{KawakitaComparisonNonLCI} and \cite{EinMustataJetSchemesAndSings}.  In this paper we explore a related phenomenon in characteristic $p > 0$ which we call \emph{$F$-adjunction} (or Frobenius-adjunction).
In particular, we prove results very similar to the parts of what was known as the adjunction conjecture of Kawamata and Shokurov, see \cite{AmbroAdjunctionConjecture}, which relates the singularities of $X$ near a center of log canonicity $W \subseteq X$ with the singularities of $W$.

Suppose that $R$ is a Gorenstein (or a sufficiently nice log-$\bQ$-Gorenstein) normal $F$-finite ring.
Then to every center of sharp $F$-purity $Q \in \Spec R$ (centers of sharp $F$-purity are characteristic $p$ analogs of centers of log canonicity)
such that $R_Q$ is $F$-pure and $R/Q$ is normal we show that there exists a canonically defined $\bQ$-divisor $\Delta_{R/Q}$ on $\Spec R/Q$ such that
the singularities of $R$ near $Q$ are ``the same'' as the singularities of $(R/Q, \Delta_{R/Q})$.

A center of sharp $F$-purity is a characteristic $p > 0$ analog of a center of log canonicity; see for example \cite[Definition 1.3]{KawamataOnFujitasFreenessConjectureFor3Folds} and \cite{SchwedeCentersOfFPurity}.  Technically speaking, a possibly non-closed point $Q \in \Spec R$ is a \emph{center of sharp $F$-purity} if, for every $R$-linear map $\phi : R^{1 \over p^e} \rightarrow R$, we have $\phi(Q^{1/p^e}) \subseteq Q$.  In particular, if $\Spec R$ is $F$-split, then $\Spec R/Q$ is compatibly split with every Frobenius splitting of $\Spec R$.  Unfortunately, there may be infinitely many different maps that one needs to check to determine whether $Q$ is a center of sharp $F$-purity.  However, when $R$ is Gorenstein and sufficiently local, there exists a ``generating'' map $\psi : R^{1/p} \rightarrow R$ such that $Q$ is a center of sharp $F$-purity if and only if $\psi(Q^{1/p}) \subseteq Q$ for this single map $\psi$, see Proposition \ref{PropUniformlyFCompatIfAndOnlyIf}.  A similar result also holds when $R$ is $\bQ$-Gorenstein with index not divisible by $p > 0$.  It is the existence of this ``generating map'' that we use to prove our results.

We will now briefly outline the construction of $\Delta$ on $R/Q$.  On any scheme $X = \Spec R$ such that $R$ is a normal local ring of characteristic $p > 0$, there is a bijection of sets
\[
\left\{ \begin{matrix}\text{Effective $\bQ$-divisors $\Delta$ such }\\\text{that $(p^e - 1)(K_X + \Delta)$ is Cartier}\end{matrix} \right\} \leftrightarrow \left\{ \text{Non-zero elements of $\Hom_{\O_X}(F^e_* \O_X, \O_X)$} \right\} \Big/ \sim
\]
where the equivalence relation on the right identifies two maps $\phi$ and $\psi$ if there is a unit $u$ such that $\phi(u \times \blank) = \psi(\blank)$; see Theorem \ref{TheoremDivisorsInduceMaps}.  Statements related to this correspondence are well known and have appeared in several previous contexts, see \cite[Proof \#2 of Theorem 3.1]{HaraWatanabeFRegFPure} and \cite{MehtaRamanathanFrobeniusSplittingAndCohomologyVanishing}.  However, we do not think it has been explicitly described in the context of $\bQ$-divisors and singularities defined by Frobenius.

With this bijection in mind, assume $(p^e - 1)K_X$ is Cartier, then the divisor $0$ on $X = \Spec R$ determines a map $\phi \in \Hom_{\O_X}(F^e_* \O_X, \O_X)$.  Setting $W = \Spec R/Q$, the map $\phi$ can be restricted to a map $\phi_Q \in \Hom_{\O_W}(F^e_* \O_W, \O_W)$ precisely because $W$ is a center of sharp $F$-purity (the map is $\phi_Q$ is non-zero because $R_Q$ is $F$-pure).  But then $\phi_Q$ corresponds to a divisor $\Delta_{R/Q}$ on $W = \Spec R/Q$.

Once we have constructed $\Delta_{R/Q}$, we can relate the singularities of $X$ and $W$.  Roughly speaking, we can do this because the $F$-singularities of $R$ (respectively, the $F$-singularities of $R/Q$) can all be defined by the images of certain $\phi \in \Hom_{\O_X}(F^e_* \O_X, \O_X)$ (respectively $\phi_Q \in \Hom_{\O_W}(F^e_* \O_W, \O_W)$).  Some of these results are summarized below:

\begin{mainthm}[Theorem \ref{ThmFirstFAdjunction}, Corollary \ref{CorRestrictionTheoremForAdjoint2}, Remark \ref{RemarkGlobalGluing}]
\label{ThmMain}
Suppose that $X$ is an integral separated normal $F$-finite noetherian scheme essentially of finite type\footnote{The essentially finite type hypothesis can be removed if one is willing to work on sufficiently small affine chart or if $X$ is the spectrum of a local ring.} over an $F$-finite field of characteristic $p > 0$.  Further suppose that $\Delta$ is an effective $\bQ$-divisor on $X$ such that $K_X + \Delta$ is $\bQ$-Cartier with index not divisible by $p$.  Let $W \subseteq X$ be an closed subscheme that satisfies the following properties:
\begin{itemize}
\item[(a)]  $W$ is integral and normal.
\item[(b)]  $(X, \Delta)$ is sharply $F$-pure at the generic point of $W$.
\item[(c)]  The ideal sheaf of $W$ is locally a center of sharp $F$-purity for $(X, \Delta)$.
\end{itemize}
Then there exists a canonically determined effective divisor $\Delta_W$ on $W$ satisfying the following properties:
\begin{itemize}
\item[(i)]  $(K_W + \Delta_W) \sim_{\bQ} (K_X + \Delta)|_W$
\item[(ii)]  Furthermore, if $(p^e - 1)(K_X + \Delta)$ is Cartier then $(p^e - 1)(K_W + \Delta_W)$ is Cartier and $(p^e - 1)\Delta_W$ is integral.
\item[(iii)]  For any real number $t > 0$ and any ideal sheaf $\ba$ on $X$ which is does not vanish on $W$, we have that
    $(X, \Delta, \ba^t)$ is sharply $F$-pure near $W$ if and only if $(W, \Delta_{W}, {\overline \ba}^t)$ is sharply $F$-pure.
\item[(iv)]  $W$ is minimal among centers of sharp $F$-purity for $(X, \Delta)$, with respect to containment of topological spaces (in other words, the ideal sheaf of $W$ is of maximal height as a center of sharp $F$-purity), if and only if $(W, \Delta_{W})$ is strongly $F$-regular.
\item[(v)]  There is a natural bijection between the centers of sharp $F$-purity of $(W, \Delta_{W})$, and the centers of sharp $F$-purity of $(X, \Delta)$ which are properly contained in $W$ as topological spaces.
\item[(vi)]  There is a naturally defined ideal sheaf $\tau_{b}(X, \nsubseteq W; \Delta, \ba^t)$, which philosophically corresponds to an analog of an adjoint ideal in arbitrary codimension, such that $\tau_{b}(X, \nsubseteq W; \Delta, \ba^t)|_{W} = \tau_{b}(W; \Delta_{W}, \overline{\ba}^t) = \text{``the big test ideal of $(R, \Delta, \overline{\ba}^t$)"}$.  Here $\ba$ and $t > 0$ are as in (iii).
\end{itemize}
\end{mainthm}
\vskip 12pt
When the center $W$ is not a normal scheme, some of these results can still be lifted to the normalization of $W$, see Proposition \ref{PropositionPropertiesOfNormalizedRestrictedDelta}.  Also see the concluding remarks to this paper.  Part (vi) should be viewed as an ultimate generalization of the $F$-restriction theorems for test ideals found in \cite{TakagiPLTAdjoint} and \cite{TakagiHigherDimensionalAdjoint}, also compare with \cite[Theorem 4.9, Remark 4.10]{HaraWatanabeFRegFPure}.

The construction of $\Delta_W$ is local and does not require $X$ to be projective.  In particular, the statement of Theorem \ref{ThmFirstFAdjunction} is ring theoretic and might be more familiar to commutative algebraists.  However the $\Delta_W$ constructed is canonical.  In particular, the $\Delta_W$ glue together to give us the result in the global setting, see Remark \ref{RemarkGlobalGluing}.

When we combine this theory with the work of Fedder, see \cite{FedderFPureRat}, we obtain the following:
\begin{atheorem}[Theorem \ref{ThmMainAppOfFedder}]
Suppose that $S$ is a regular $F$-finite ring such that $F^e_*S$ is a free $S$ module (for example, if $S$ is local) and that $R = S/I$ is a quotient that is a normal domain.  Further suppose that $\Delta_R$ is an effective $\bQ$-divisor on $\Spec R$ such that $\Hom_R(F^e_* R((p^e - 1) \Delta ), R)$ is a rank one free $F^e_* R$-module (for example, if $R$ is local and $(p^e - 1)(K_R + \Delta$) is Cartier).  Then there exists an effective $\bQ$-divisor $\Delta_S$ on $\Spec S$ such that:
\begin{itemize}
\item[(a)]  $(p^e - 1)(K_S + \Delta_S)$ is Cartier.
\item[(b)]  $I$ is $(\Delta_S, F)$-compatible and $(S, \Delta_S)$ is sharply $F$-pure at the minimal associated primes of $I$ (that is, the generic points of $\Spec S/I$).
\item[(b)]  $\Delta_S$ induces $\Delta_R$ as in the Main Theorem.
\end{itemize}
\end{atheorem}
\hskip -12pt We do not know of any similar result proved in characteristic $0$ (except when $R$ is a complete intersection, see \cite{EinMustataYasuda}).   The $\Delta_S$ in Theorem \ref{ThmMainAppOfFedder} is not canonically determined and therefore we do not see how to globalize this statement.

We also prove the following result.
\begin{ntheorem} [Corollary \ref{CorCompatFrobeniusSplitIdealsAreCenters}, Remark \ref{RemarkGlobalGluing}]
Suppose that $X$ is a normal variety of finite type over an $F$-finite field $k$.  Suppose that $\phi : F^e_* \O_X \rightarrow \O_X$ is a (global) splitting of Frobenius.  Then there exists an effective divisor $\Delta$ on $X$ (determined uniquely by $\phi$) such that
\begin{itemize}
\item[(1)]  $K_X + \Delta \sim_{\bQ} 0$,
\item[(2)]  $(X, \Delta)$ is sharply $F$-pure,
\item[(3)]  The irreducible subvarieties compatibly split by $\phi$ coincide exactly with the centers of sharp $F$-purity of $(X, \Delta)$.
\end{itemize}
\end{ntheorem}
\hskip -12pt  Since centers of sharp $F$-purity are closely related to centers of log canonicity, the previous result should be viewed as a link between compatibly split subvarieties and centers of log canonicity (of log Calabi-Yau pairs).

Finally, also using these ideas, we prove that there are only finitely many centers of sharp $F$-purity for a sharply $F$-pure triple $(R, \Delta, \ba_{\bullet})$ (the case when $R$ is a local ring was done in \cite{SchwedeCentersOfFPurity} using the techniques of \cite{EnescuHochsterTheFrobeniusStructureOfLocalCohomology} or \cite{SharpGradedAnnihilatorsOfModulesOverTheFrobeniusSkewPolynomialRing}).  Here $\ba_{\bullet}$ is a graded system of ideals; see  \cite{HaraACharacteristicPAnalogOfMultiplierIdealsAndApplications} and \cite{SchwedeCentersOfFPurity}.
\begin{ctheorem}[Theorem \ref{ThmFinitelyManyCenters}]
{\it  If $(R, \Delta, \ba_{\bullet})$ is sharply $F$-pure, then there are at most finitely many centers of sharp $F$-purity.  }
\end{ctheorem}
\hskip -12pt  This also implies that if $X$ is noetherian (although not necessarily affine) and $(X, \Delta)$ is locally sharply $F$-pure, then there are at most finitely many centers of sharp $F$-purity.  This is the analog of the statement that ``if $(X, \Delta)$ is log canonical, there exists at most finitely many centers of log canonicity''.  Another implication of this is that for a globally $F$-split variety, there are at most finitely many subschemes compatibly split with any given splitting, see Corollary \ref{CorFinitelyManyCompatiblySplit}.  In the case of a local ring, similar results have been obtained in \cite{EnescuHochsterTheFrobeniusStructureOfLocalCohomology} and in \cite{SharpGradedAnnihilatorsOfModulesOverTheFrobeniusSkewPolynomialRing}, also see \cite[Corollary 5.2]{SchwedeCentersOfFPurity}.  Finally, essentially the same result has been independently obtained by Kumar and Mehta, see \cite{KumarMehtaFiniteness}.

We conclude this paper with comparison of $\Delta_{R/Q}$ with related constructions which have been considered in characteristic zero (in particular, the aforementioned ``different'').  We then consider what happens if we normalize $R/Q$ (in case $R/Q$ is not normal).
We conclude with several further remarks and questions.  In particular see Remark \ref{RemarkGlobalGluing} where a global version of the ideas of this paper are briefly discussed.   

\vskip 12pt
\hskip -12pt{\it Acknowledgments:}

The author would like to thank Florin Ambro, Manuel Blickle, Mel Hochster, Karen Smith, Shunsuke Takagi, and Wenliang Zhang for several valuable discussions.  The author would also like to thank Shunsuke Takagi and the referee for carefully reading an earlier draft and providing several useful comments.

\section{Preliminaries and notation}
\label{SectionFSingularities}

Throughout this paper, all schemes and rings are noetherian, excellent, reduced and of characteristic $p > 0$.  We also assume that all rings $R$ (and schemes $X$) have locally normalized dualizing complexes, $\omega_R^{\mydot}$ (respectively $\omega_X^{\mydot}$), see \cite{HartshorneResidues}.  In fact, little is lost if one only considers rings that are of essentially finite type over a perfect field.  Since we are primarily concerned with the affine or local setting, we will freely switch between the notation corresponding to a ring $R$ and the associated scheme $X = \Spec R$.  If $X = \Spec R$ and $R$ is reduced, then we will use $k(X) = k(R)$ to denote the total field of fractions of $R$.  If $D$ is a divisor on $X = \Spec R$, we will mix notation and use $R(D)$ to denote the global sections of $\O_X(D)$.  Furthermore, we will often use $F^e_* M$ to denote an $R$-module $M$ viewed as an $R$-module via the $e$-iterated Frobenius, that is $r.x = r^{p^e} x$ (informally, this is just restriction of scalars).  In particular, when $R$ is reduced $F^e_* R$ is just another notation for $R^{1 \over p^e}$.  The reason for this notation is that if $F^e : X \rightarrow X$ is the $e$-iterated Frobenius, then $F^e_* \O_X$ is just the sheaf associated to $R^{1 \over p^e}$.

We briefly review some properties of Weil divisors on normal schemes, compare with \cite[Chapter II, Section 6]{Hartshorne}, \cite{HartshorneGeneralizedDivisorsOnGorensteinSchemes} and \cite[Chapter 7]{BourbakiCommutativeAlgebraTranslation}.  Recall that on a normal scheme $X$, a \emph{Weil divisor} is finite formal sum of reduced and irreducible subschemes of codimension 1, and a \emph{prime divisor} is a single irreducible subscheme of codimension 1.   So if $X = \Spec R$, the Weil divisors carry the same information as formal sums of height one prime ideals.  A \emph{$\bQ$-Divisor} is an element of $\{\text{group of Weil divisors} \} \tensor_{\bZ} \bQ$, it can also be viewed as a finite formal sum $\sum a_i D_i$ where the $a_i \in \bQ$ and the $D_i$ are prime divisors.  See \cite{KollarMori} for basic facts about $\bQ$-divisors from this point of view.  A $\bQ$-divisor for which all the $a_i$ are integers is called an \emph{integral divisor} (in other words, an integral divisor is a $\bQ$-divisor that is also a Weil divisor).  A $\bQ$-divisor is called \emph{$\bQ$-Cartier} if there exists an integer $m > 0$ such that $mD$ is an integral Cartier divisor.  A $\bQ$-divisor is called \emph{$m$-Cartier} if $mD$ is an integral Cartier divisor.  A divisor (respectively a $\bQ$-divisor) $D = \sum a_i D_i$ is called \emph{effective} if each of the $a_i$ are non-negative integers (respectively, non-negative rational numbers).

Since $X$ is normal, for each prime divisor $D$ on $X$, there is an associated discrete valuation $v_D$ at the generic point of $D \subset X$.
Then, for any non-degenerate element $f \in k(X)$ (an element is \emph{non-degenerate} if it is non-zero on each generic point of $X = \Spec R$),
there is a divisor $\Div f$ which is defined as $\Div f = \sum_{D \subset X} v_D(f) D$.
Recall that associated to any divisor $D$ on $X = \Spec R$ there is a coherent sheaf $\O_X(D)$ whose global sections
 we will denote by $R(D)$.  Recall that the sheaf $R(D)$ is reflexive with respect to $\Hom_R(\blank, R)$.  

For the convenience of the reader, we record some useful properties of reflexive sheaves that we will use without comment.

\begin{proposition} \cite{Hartshorne}, \cite[Proposition 1.11, Theorem 1.12]{HartshorneGeneralizedDivisorsOnGorensteinSchemes}
Suppose that $R$ is a normal ring and suppose that $M$ and $N$ are finitely generated $R$-modules.  Then:
\begin{itemize}
\item[(1)]  $M$ is reflexive (that is, the natural map $M \rightarrow \Hom_R(\Hom_R(M, R), R) = (M^{\vee})^{\vee}$ is an isomorphism) if and only if $M$ is S2.
\item[(2)]  $\Hom_R(M, R) = M^{\vee}$ is reflexive.
\item[(3)]  If $R$ is of characteristic $p$ and $F$-finite (see Definition \ref{DefnFFinite}), then $M$ is reflexive if and only if $F^e_* M$ is reflexive.
\item[(4)]  If $N$ is reflexive, then $\Hom(M, N)$ is also reflexive.
\item[(5)]  Suppose $M$ is reflexive, that $X = \Spec R$ and $Z \subset X$ is a closed subset of codimension 2.  Set $U$ to be $X \setminus Z$ and let $i : U \rightarrow X$ be the inclusion.  Then $i_* (M|_U) \cong M$.
\item[(6)]  With notation as in (5), the restriction map to $U$ induces an equivalence of categories from reflexive coherent sheaves on $X$ to reflexive coherent sheaves on $U$.
\end{itemize}
\end{proposition}

\begin{proposition}
\label{PropSectionsAreSameAsEmbeddings}  \cite[Proposition 2.9]{HartshorneGeneralizedDivisorsOnGorensteinSchemes}, \cite[Remark 2.9]{HartshonreGeneralizedDivisorsAndBiliaison}
Suppose that $X$ is a normal scheme and $D$ is a divisor on $X$.
Then, there is a one-to-one correspondence between effective divisors linearly equivalent to $D$ and non-degenerate sections\footnote{A section is called \emph{non-degenerate} if it is non-zero at the generic point of every irreducible component of $X$.}  $s \in \Gamma(X, \O_X(D))$ modulo multiplication by units in $H^0(X, \O_X)$.
\end{proposition}

\begin{definition}
If $X$ is equidimensional, then we set $\omega_X$ to be $h^{-\dim X}(\omega_X^{\mydot})$ and call it the \emph{canonical module} of $X$.
If, in addition, $X$ is normal, then $\omega_X$ is a rank 1 reflexive sheaf and so it corresponds to an integral divisor class.
A divisor $D$ such that $\O_X(D) \cong \omega_X$ is called a \emph{canonical divisor of $X$} and is denoted by $K_X$.
\end{definition}

\begin{remark}
If $X$ is not normal but instead Gorenstein in codimension 1 (G1) and S2, then one can still view $\omega_X$ as a divisor class (technically as an  ``almost Cartier divisor'' / ``Weil divisorial subsheaf'', see \cite{HartshorneGeneralizedDivisorsOnGorensteinSchemes} and \cite{KollarFlipsAndAbundance}).  Most of the results of this paper generalize to pairs $(X, \Delta)$ where $X$ is G1 and S2 and $\Delta$ is an element from $\{\text{almost Cartier divisors}\} \tensor \bQ$.  However, there are several technical complications which we feel obscure the main points of this paper and so we will not work in this generality.  In particular, one can have two different almost Cartier divisors / Weil divisorial subsheaves $D$ and $E$ such that $2D = 2E$, see \cite[Page 172]{KollarFlipsAndAbundance}.  Because of this, for a $\bQ$-Weil divisorial subsheaf $D$, $\O_X(D)$ is not well defined.  There are ways around this issue, although statements like Theorem \ref{TheoremMapsInduceDivisors}(e,f) and the definition of sharply $F$-pure pairs would need to be amended.  Another option is to do something similar to what is suggested in Remark \ref{RemarkFinalNonNormal}.
\end{remark}

\begin{definition}
A \emph{pair} $(X, \Delta)$ is the combined information of a normal scheme $X$ and an effective $\bQ$-divisor $\Delta$.  A \emph{triple} $(X, \Delta, \ba^t)$ is the combined information of a pair $(X, \Delta)$, an ideal sheaf $\ba \subseteq \O_X$ which on every chart $U = \Spec R$ satisfies $\ba|_U \cap R^{\circ} \neq \emptyset$, and a positive real number $t > 0$.  If $X = \Spec R$, then we will sometimes write $(R, \Delta)$ instead of $(X, \Delta)$.
\end{definition}

Now we define $F$-singularities, singularities defined by the action of Frobenius.  These are classes of singularities associated with tight closure theory, see \cite{HochsterHunekeTC1}, that are good analogs of singularities from the minimal model program, see for example \cite{KollarMori}.

\begin{definition}
\label{DefnFFinite}
We say that a ring $R$ of positive characteristic $p > 0$ is \emph{$F$-finite} if $F_* R = R^{1 \over p}$ is finite as an $R$-module.
\end{definition}

Throughout the rest of this paper, \emph{all} rings will be assumed to be $F$-finite.  This is not too restrictive of an assumption since any ring essentially of finite type over a perfect field is $F$-finite, see \cite[Lemma 1.4]{FedderFPureRat}.

\begin{definition} \cite{HochsterRobertsFrobeniusLocalCohomology}, \cite{HochsterHunekeTightClosureAndStrongFRegularity}, \cite{HaraWatanabeFRegFPure}, \cite{SchwedeSharpTestElements}
\label{DefnStronglyFRegularSharplyFPure}
Suppose that $(R, \bm)$ is a local ring.  We say that a triple $(R, \Delta, \ba^t)$ is \emph{sharply $F$-pure} if there exists an integer $e > 0$, an element $a \in \ba^{\lceil t(p^e - 1) \rceil}$ and a map $\phi \in \Hom_R(F^e_* R(\lceil (p^e - 1)\Delta \rceil), R)$ such that $\phi(F^e_* (a R) ) = R$.  Here $F^e_* (a R) \subseteq F^e_* R(\lceil (p^e - 1)\Delta \rceil)$.  If $\Delta = 0$ and $\ba = R$, then we call the sharply $F$-pure triple $(R, \Delta, \ba^t)$ (or simply the ring $R$) \emph{$F$-pure}.

Again, assuming $R$ is local, a triple $(R, \Delta, \ba^t)$ is called \emph{strongly $F$-regular} if for every $c \in R^{\circ}$ there is an integer $e > 0$,  an element $a \in \ba^{\lceil t(p^e - 1) \rceil}$, and a map $\phi \in \Hom_R(F^e_* R(\lceil (p^e - 1)\Delta \rceil), R)$ such that $\phi(F^e_* (c a R) ) = R$.

If $X$ is any scheme (for example $X = \Spec R$ where $R$ is a non-local ring), then a triple $(X, \Delta, \ba^t)$ is called \emph{sharply $F$-pure} (respectively, \emph{strongly $F$-regular}) if for every closed point\footnote{If the condition holds at the closed points, then it also holds at the non-closed points.} $x \in X$, the localized triple $(\O_{X,x}, \Delta|_{\Spec \O_{X,x}}, \ba^t_{x})$ is sharply $F$-pure (respectively, strongly $F$-regular).
\end{definition}

\begin{remark}
\label{RemarkBetterDefnStrongFRegSharpFPure}
In the case that $R$ is a non-local ring, these definitions of strong $F$-regularity and sharp $F$-purity are slightly more general than the ones given in \cite{TakagiInversion}, in \cite{TakagiWatanabeFPureThresh}, in \cite{SchwedeSharpTestElements}, or in \cite{SchwedeCentersOfFPurity}.  Previously, a triple $(R, \Delta, \ba^t)$ (with $R$-not necessarily local) was called strongly $F$-regular (respectively sharply $F$-pure) if it satisfied the ``local ring'' version of the condition stated above.  In the case that $\ba = R$ (or more generally, if $\ba$ is principal) then the various notions coincide (regardless of the $\Delta$).  The problem is that it is not clear whether a triple $(R, \Delta, \ba^t)$ is strongly $F$-regular (respectively sharply $F$-pure) if and only if it is strongly $F$-regular (respectively sharply $F$-pure) after localizing at every maximal ideal.
\end{remark}

\begin{remark}
Suppose that $R$ is local and that $(R, \Delta, \ba^t)$ is sharply $F$-pure and that $e$ is as in the above definition, then for every integer $n > 0$ there exists a $\phi_n \in \Hom_R(F^{ne}_* R(\lceil (p^{ne} - 1) \Delta \rceil), R)$ such that $1 \in \phi_n(  F^{ne}_* \ba^{\lceil t(p^{ne} - 1) \rceil})$.  This is follows from the same argument as in \cite[Lemma 2.8]{SchwedeCentersOfFPurity} or \cite[Proposition 3.3]{SchwedeSharpTestElements}.
\end{remark}

\begin{remark}
Sharply $F$-pure singularities are a characteristic $p>0$ analog of log canonical singularities, see \cite{HaraWatanabeFRegFPure} and \cite{SchwedeSharpTestElements}.  Strongly $F$-regular singularities are a characteristic $p>0$ analog of Kawamata log terminal singularities, see \cite{HaraWatanabeFRegFPure}.  There are also good analogs of purely log terminal singularities that we will not discuss here, see \cite{TakagiPLTAdjoint}.
\end{remark}

\begin{definition} \cite{HochsterHunekeTC1}, \cite{HaraTakagiOnAGeneralizationOfTestIdeals}, \cite{SchwedeSharpTestElements}, \cite{SchwedeCentersOfFPurity}
Suppose that $(R, \Delta, \ba^t)$ is a triple.  An element $c \in R^{\circ}$ is called a big sharp test element for $(R, \Delta, \ba^t)$ if for all modules $N \subseteq M$ and all $z \in N^{* \Delta, \ba^t}_M$, one has that $c \ba^{\lceil t(p^e - 1) \rceil} z^{p^e} \subseteq N^{[p^e]\Delta}_M$ for all $e \geq 0$.

For the definition of tight closure with respect to such a triple (and an explanation of the notation above), see \cite[Definition 2.14 ]{SchwedeCentersOfFPurity}.  Also compare with \cite{HaraYoshidaGeneralizationOfTightClosure}, \cite{TakagiInterpretationOfMultiplierIdeals}, and \cite{TakagiPLTAdjoint}.
\end{definition}

If $R$ is reduced and $F$-finite, then there always exists a big sharp test element for any triple $(R, \Delta, \ba^t)$.

\begin{definition}  \cite{HochsterHunekeTC1}, \cite{HaraTakagiOnAGeneralizationOfTestIdeals}, \cite{LyubeznikSmithCommutationOfTestIdealWithLocalization}, \cite{HochsterFoundations}
The \emph{big test ideal} of a triple $(R, \Delta, \ba^t)$, denoted $\tau_b(R; \Delta, \ba^t)$, is defined as follows:
Set $E = \oplus_{\bm \in \bm-\Spec R} E_{R/\bm}$, where $E_{R/\bm}$ is the injective hull of $R/\bm$.  Then
\[
\tau_b(R; \Delta, \ba^t) := \Ann_R 0^{* \Delta \ba^t}_{E} = \bigcap_{\bm} \Ann_R 0^{* \Delta, \ba^t}_{E_{R/\bm}}.
\]
\end{definition}

\begin{remark}
Big test ideals are characteristic $p > 0$ analogs of multiplier ideals, see \cite{SmithMultiplierTestIdeals}, \cite{HaraInterpretation}, \cite{TakagiInterpretationOfMultiplierIdeals} and \cite{HaraYoshidaGeneralizationOfTightClosure}.
\end{remark}

\begin{remark}
 In \cite{SchwedeCentersOfFPurity}, the author defined the big test ideal $\tau_b(R; \Delta, \ba^t)$ in a somewhat different way, essentially using the criterion for the big test ideal found in \cite[Lemma 2.1]{HaraTakagiOnAGeneralizationOfTestIdeals}.  While we will not state that definition here, we note that the big test ideal of \cite{SchwedeCentersOfFPurity} was an ideal $J$ of $R$ which, when localized at any $\bm$, coincided with $\Ann_{R_\bm} 0^{* \Delta \ba^t}_{E_{R/\bm}}$.  We now explain why such a $J$ agrees with $\tau_b(R; \Delta, \ba^t)$.  Note that this $J$ is contained in each $\Ann_{R} 0^{* \Delta \ba^t}_{E_{R/\bm}}$, and so $J \subseteq \Ann_R 0^{* \Delta \ba^t}_{E}$.  Conversely, we see that $\tau_b(R; \Delta, \ba^t) R_{\bm} \subseteq \Ann_{R_\bm} 0^{* \Delta \ba^t}_{E_{R/\bm}} \subseteq J_{\bm}$ which completes the proof.
\end{remark}

\begin{definition} \cite{SchwedeCentersOfFPurity}
\label{DefnUniformlyFCompat}
An ideal $I \subseteq R$ is said to be \emph{$F$-compatible with respect to $(R, \Delta, \ba^t)$} or equivalently \emph{uniformly $(\Delta, \ba^t, F)$-compatible} or simply {\emph{$F$-compatible}} if the context is clear, if for every $e > 0$, every $a \in \ba^{\lceil t(p^e - 1) \rceil}$ and every map $\phi \in \Hom_R(F^e_* R(\lceil t(p^e - 1) \Delta \rceil), R)$, we have $\phi(F^e_* a I) \subseteq I$.  A \emph{prime} ideal $Q$ which is $F$-compatible with respect to $(R, \Delta, \ba^t)$ is called a \emph{center of sharp $F$-purity for $(R; \Delta, \ba^t)$}, or simply a \emph{center of $F$-purity} if the context is clear.  We will also often abuse notation and call the subscheme $W := \Spec R/Q \subseteq \Spec R =: X$ a \emph{center of $F$-purity} as well.
\end{definition}

\begin{remark}
Centers of sharp $F$-purity are characteristic $p > 0$ analogs of centers of log canonicity.  In particular, any center of log canonicity reduced from characteristic $0$ to characteristic $p \gg 0$ is a center of sharp $F$-purity, see \cite[Theorem 6.7]{SchwedeCentersOfFPurity}.
\end{remark}

The following results on $F$-compatible ideals will be used later.

\begin{lemma}  \cite{SchwedeCentersOfFPurity}
\label{LemmaPropertiesOfUniformlyFCompatible}
Consider a triple $(R, \Delta, \ba^t)$ (recall all rings are assumed $F$-finite).  Then the following properties of $F$-compatible ideals are satisfied.
\begin{itemize}
\item[(1)]  Any (ideal-theoretic) intersection of $F$-compatible ideals is $F$-compatible.
\item[(2)]  Any (ideal-theoretic) sum of $F$-compatible ideals is $F$-compatible.
\item[(3)]  The radical of an $F$-compatible ideal is $F$-compatible.
\item[(4)]  The big test ideal $\tau_b(R; \ba^t, \Delta)$ is the unique smallest $F$-compatible ideal that has non-trivial intersection with $R^{\circ}$.
\item[(5)]  The minimal primes of a radical $F$-compatible ideal are also $F$-compatible.
\item[(6)]  A pair $(R, \Delta)$ is strongly $F$-regular if and only if it has no centers of sharp $F$-purity besides the minimal primes of $R$.
\end{itemize}
\end{lemma}
A version of Lemma \ref{LemmaPropertiesOfUniformlyFCompatible}(6) is true also for triples $(R, \Delta, \ba^t)$.  Although in that case, one must use the ``new'' strong $F$-regularity condition, see Remark \ref{RemarkBetterDefnStrongFRegSharpFPure}.  In particular, \cite[Corollary 4.6]{SchwedeCentersOfFPurity} is probably not correct as stated.  It should say
\begin{quote}``$(R, \Delta, \ba_{\bullet})$ is strongly $F$-regular after localizing at every maximal ideal of $R$ if and only if $(R, \Delta, \ba_{\bullet})$ has no centers of sharp $F$-purity besides the minimal primes of $R$.''
\end{quote}
Thus the original statement of \cite[Corollary 4.6]{SchwedeCentersOfFPurity} is correct if one uses the definition of strong $F$-regularity from Definition \ref{DefnStronglyFRegularSharplyFPure}.  We believe this is the only instance of the issue described in Remark \ref{RemarkBetterDefnStrongFRegSharpFPure} causing a misstatement in the paper \cite{SchwedeCentersOfFPurity} (although several results can be strengthened if one uses the ``new'' definition).

\section{Relation between Frobenius and boundary divisors}
\numberwithin{equation}{theorem}
In this section we'll describe a correspondence between maps $\phi : F^e_* \O_X \rightarrow \O_X$ and $\bQ$-divisors $\Delta$ such that $K_X + \Delta$ is $\bQ$-Cartier (with index not divisible by $p > 0$).  Statements closely related to this correspondence have appeared in several previous contexts, see \cite[Proof \#2 of Theorem 3.1]{HaraWatanabeFRegFPure} and \cite{MehtaRamanathanFrobeniusSplittingAndCohomologyVanishing}, and were known to experts.  However, we do not think the correspondence has been explicitly written from a $\bQ$-divisor perspective.  As before, in this section we are assuming that $X$ is the spectrum of a normal $F$-finite ring $R$ with a locally normalized dualizing complex $\omega_R^{\mydot}$.

Roughly speaking, the correspondence goes like this.  Suppose $R$ is a local ring and set $X = \Spec R$:
\begin{itemize}
 \item{} Given a $\phi \in \Hom_R(F^e_* R, R)$, this is the same as
 \item{} choosing a map (of $F^e_* R$-modules) $F^e_* R \rightarrow \Hom_R(F^e_* R, R)$ sending $1$ to $\phi$, which is the same as
 \item{} an effective Weil divisor $D$ such that $\O_X(D) \cong \O_X((1 - p^e)K_X)$ (note $F^e_* \O_X((1 - p^e)K_X) \cong \Hom_R(F^e_* R, R)$), which is the same as
 \item{} an effective $\bQ$-divisor $\Delta$ where we set $\Delta = {1 \over p^e - 1}D$.
\end{itemize}

The expert reader might wonder why we divide by $p^e - 1$ in the final step (and thus produce a $\bQ$-divisor).  It turns out that for the purposes of $F$-singularities, composing $\phi$ with itself (ie, $\phi \circ F^e_* \phi$) is harmless, see Section \ref{SectionApplicationsToCenters} below.  Thus by dividing by $p^e - 1$ we are normalizing our divisor with respect to composition; see Theorem \ref{TheoremMapsInduceDivisors}(e).

In order to make this correspondence precise and in order to be able to use it, we first need the following observations about maps $F^e_* \O_X \rightarrow \O_X$ (which of themselves are of independent interest).  Lemma \ref{LemmaHomIsCanonical} is well known to experts, see \cite{FedderFPureRat}, \cite{MehtaRamanathanFrobeniusSplittingAndCohomologyVanishing}, \cite{MehtaSrinivasFPureSurface} and \cite[Lemma 3.4]{HaraWatanabeFRegFPure}, however the proof is short, so we include it for the convenience of the reader.

\begin{lemma}
\label{LemmaHomIsCanonical}
Suppose that $(X, \Delta)$ is a pair such that $(p^e - 1)(K_X + \Delta)$ is a Cartier divisor.  Then $\sHom_{\O_X}(F^e_* \O_X( (p^e - 1) \Delta), \O_X)$ is an invertible sheaf when viewed as an $F^e_* \O_X$-module.
\end{lemma}
\begin{proof}
It is enough to verify this locally, so we may assume that $X$ is the spectrum of a local ring.  Then observe that
\[
\begin{split}
\sHom_{\O_X}(F^e_* \O_X( (p^e - 1) \Delta), \O_X)  \cong \sHom_{\O_X}(F^e_* \O_X( (p^e - 1) \Delta + p^e K_X), \omega_X) \cong \\
F^e_* \sHom_{\O_X}(\O_X( (p^e - 1) \Delta + p^e K_X), \omega_X) \cong F^e_* \O_X((1 - p^e)(K_X + \Delta)) \cong F^e_* \O_X.
\end{split}
\]
\end{proof}

\begin{remark}
We will often view $\sHom_{\O_X}(F^e_* \O_X( (p^e - 1) \Delta), \O_X)$ as an $F^e_* \O_X$-submodule of $\sHom_{\O_X}(F^e_* \O_X, \O_X)$.  
\end{remark}

\begin{remark}
\label{RemHaveToBeCareful}
 For an arbitrary normal (non-local) $F$-finite scheme $X$, we do not know if one always has
\begin{equation}
\label{EqnMustBeCareful}
\sHom_{\O_X}(F^e_* \O_X( (p^e - 1) \Delta), \O_X) \cong \O_X((1 - p^e)(K_X + \Delta)).
\end{equation}
In the non-local case, if one is following the proof of Lemma \ref{LemmaHomIsCanonical}, one should write
\[
 \sHom_{\O_X}(F^e_* \O_X( (p^e - 1) \Delta + p^e K_X), \omega_X) \cong \\
F^e_* \sHom_{\O_X}(\O_X( (p^e - 1) \Delta + p^e K_X), (F^e)^! \omega_X).
\]
The module $(F^e)^! \omega_X = \Hom_{\O_X}(F^e_* \O_X, \omega_X)$ is a canonical module on $X$, but these are only unique up to tensoring with an invertible sheaf.  In the local case, tensoring with an invertible sheaf does nothing (and so $\omega_X$ is unique up to isomorphism -- multiplication by a unit).  Likewise, if $X$ is of essentially finite type over an $F$-finite field, it is easy to see that $(F^e)^! \omega_X$ can be identified with $\omega_X$ (again, non-canonically, but up to multiplication by a unit of $H^0(X, \O_X)$).  Of course, by passing to a sufficiently small affine chart, we can always assume that Equation \ref{EqnMustBeCareful} is satisfied.  In fact, it may be that Equation \ref{EqnMustBeCareful} always holds.
\end{remark}

The previous result also implies the following when interpreted using Fedder's criterion, see \cite{FedderFPureRat}.
\begin{corollary}
\label{CorDescriptionOfFedderColon}
Suppose that $(R, \bm)$ is a quasi-Gorenstein normal local ring (respectively, a $\bQ$-Gorenstein local ring whose index is a factor of $p^d -1$).  Further suppose that we can write $R = S/I$ where $S$ is an $F$-finite regular local ring.  Then for each $e > 0$ (respectively for each $e = n d$, $n > 0$) there exists an element $f_e \in R$ so that $(I^{[p^e]} : I) = I^{[p^e]} + (f_e)$.
\end{corollary}
\begin{proof}
Simply note that $F^e_* (I^{[p^e]} : I) \Hom_S(F^e_* S, S) / F^e_* I^{[p^e]} \cong \Hom_R(F^e_* R, R)$ by \cite[Lemma 1.6]{FedderFPureRat}.  The quasi-Gorenstein or $\bQ$-Gorenstein assumption implies that the right side of the equation is a free rank-one $F^e_* R$-module.
\end{proof}

\begin{remark}
If one fixes a generator $T$ of $\Hom_S(F^e_* S, S)$, one can then view the element $f_e$ as an $S$-module map $F^e_* S \rightarrow S$ that sends $F^e_* I$ into $I$.
\end{remark}


\begin{observation}
Suppose (in the situation of Lemma \ref{LemmaHomIsCanonical}) that $X$ is the spectrum of a local ring, that $\Delta = 0$, and that $\O_X((p^e - 1) K_X)$ is a free rank-one $F^e_* \O_X$-module.  Therefore,  $\sHom_{\O_X}(F^e_* \O_X, \O_X)$ has a generator $T$.  If one composes $T$ with its pushforward $F^e_* T : F^{2e}_* \O_X \rightarrow F^e_* \O_X$, one obtains a map
\begin{equation}
\label{CompositionEquation}
T_{2e} = T \circ F^e_* T : F^{2e}_* \O_X \rightarrow \O_X.
\end{equation}
One can then ask whether that composition is a generator of the rank-one locally free $F^{2e}_* R$-module $\sHom_{\O_X}(F^{2e}_* \O_X, \O_X)$?  What can be said in the case that $\Delta \neq 0$?  It turns out that the composition is indeed a generator (and in the case when $\Delta \neq 0$ as well).  One can prove this using local duality, however it is no more difficult (and certainly more satisfying) to prove it directly.  First however, let us compute a specific example.
\end{observation}

\begin{example}
\label{ExampleRegularCase}
Consider the case when $X = \Spec \bF_p[x_1, \ldots, x_n] = \Spec R$ and choose $T_e$ to be the generator of $\sHom_{R}(F^e_*R, R)$ of the form
\[
T_e(x_1^{l_1} x_2^{l_2} \dots x_n^{l_n}) = \left\{ \begin{array}{l l} 1, & \text{if }  l_1 = l_2 = \ldots = l_n = p^e - 1 \\ 0, & \text{whenever $l_i \leq p^e - 1$ for all $i$ and $l_i < p^e - 1$ for some $i$} \end{array} \right.
\]
Now consider $T_e \circ F^e_* T_e$, we claim it is equal to $T_{2e}$.  Consider a monomial $m = x_1^{l_1} x_2^{l_2} \dots x_n^{l_n}$ such that $l_i \leq p^{2e} - 1$.  We can write
\[
m = (x_1^{k_1})^{p^e}(x_1^{j_1}) (x_2^{k_2})^{p^e}(x_2^{j_2}) \dots (x_n^{k_n})^{p^e}(x_n^{j_n}),
\]
where $k_i, j_i < p^e$ are integers.
This implies that $T_e(F^e_* T_e(m)) =  T_e(x_1^{k_1} \dots x_n^{k_n} T_e(x_1^{j_1} \dots x_n^{j_n}))$.  The claim is then easily verified since $p^e(p^e - 1) + (p^e - 1) = (p^{2e} - 1)$.
\end{example}

\begin{remark}
In the context of Example \ref{ExampleRegularCase}, it follows that $T_e(F^e_* I) = I^{[1/p^e]}$, where $I^{[1/p^e]}$ is the smallest ideal $J$ such that $I \subseteq J^{[p^e]}$; see \cite{BlickleMustataSmithDiscretenessAndRationalityOfFThresholds}.  This was well known to experts.

\end{remark}

In fact, Example \ref{ExampleRegularCase} above is a special case of the following lemma (that is known to experts) which uses $\Hom$-$\tensor$ adjointness.  For example, it is closely related to \cite[Appendix F.17(a)]{KunzKahlerDifferentials}.

\begin{lemma}
\label{LemmaCompositionGivesGens}
Suppose that $R \rightarrow S$ is a finite map of rings such that $\Hom_R(S, R)$ is isomorphic to $S$ as an $S$-module.
Further suppose that $M$ is a finite $S$-module.

Then the natural map
\begin{equation}
\label{EquationCompositionMap}
 \Hom_S(M, S) \times \Hom_R(S, R) \to \Hom_R(M, R)
\end{equation}
induced by composition is surjective.
\end{lemma}
\begin{proof}
First, set $\alpha$ to be a generator (as an $S$-module) of $\Hom_R(S, R)$.  Suppose we are given $f \in \Hom_R(M, R) \cong \Hom_R(M \tensor_S S, R)$.  We wish to write it as a composition.

Using adjointness, this $f$ induces an element $\Phi(f) \in \Hom_S(M, \Hom_R(S, R))$.  Just as with the usual Hom-Tensor adjointness, we define $\Phi(f)$ by the following rule:
\[
 (\Phi(f)(t))(s) = f(t \tensor s) = f(st) \text{ for $t \in M$, $s \in S$}.
\]
Therefore, since $\Hom_R(S, R)$ is generated by $\alpha$, for each $f$ and $t \in M$ as above, we associate a unique element
$a_{f, t} \in S$ with the property that $(\Phi(f)(t))(\blank) = \alpha(a_{f, t} \blank)$.

Thus using the isomorphism $\Hom_R(S, R) \cong S$, induced by sending $\alpha$ to $1$, we obtain a map $\Psi : \Hom_R(M, R) \rightarrow \Hom_S(M, S)$ given by $\Psi(f)(t) = a_{f, t}$.  

We now consider $\alpha \circ (\Psi(f))$.
However,
\[
\alpha(\Psi(f)(t)) = \alpha(a_{f, t}) = (\Phi(f)(t))(1) = f(t).
\]
Therefore $f = \alpha \circ (\Psi(f))$ and we see that the map (\ref{EquationCompositionMap}) is surjective as desired.
\end{proof}

We need a certain variant of this in the context of pairs.

\begin{corollary}
\label{CorEveryMapComposesWithAGenerator}
Suppose that $(X, \Delta)$ is a pair and that $K_X + \Delta$ is $(p^e - 1)$-Cartier.  Then for every $d > 0$ the natural map $\Psi$,
\begin{align*}
& \sHom_{F^{e}_*\O_X} (F^{e+d}_* \O_X( \lceil (p^d - 1) \Delta \rceil), F^e_* \O_X) \\
& \hspace{30pt} \tensor_{F^e_* \O_X} \sHom_{\O_X}(F^e_* \O_X( (p^e - 1) \Delta), \O_X) \\
& \cong \sHom_{F^{e}_*\O_X}(F^{e+d}_* \O_X( \lceil (p^{e + d} - 1) \Delta \rceil), F^e_* \O_X( (p^e - 1) \Delta) ) \\
& \hspace{30pt} \tensor_{F^e_* \O_X} \sHom_{\O_X}(F^e_* \O_X( (p^e - 1) \Delta), \O_X)\\
& \rightarrow \sHom_{\O_X}(F^{e+d}_* \O_X( \lceil (p^{e + d} - 1) \Delta \rceil), \O_X )
\end{align*}
induced by composition, is an isomorphism.

In other words, locally, every map $\phi : F^{e+d}_* \O_X( \lceil (p^{e + d} - 1) \Delta \rceil) \rightarrow \O_X$ factors through some scaling of the (local) $F^e_* \O_X$-generator of
\[
\sHom_{\O_X}(F^e_* \O_X( (p^e - 1) \Delta), \O_X).
\]
\end{corollary}
\begin{proof}
Notice that the map $\Psi$ we are considering is a map of rank-one reflexive (that is, rank-one S2) $F^{e+d}_* \O_X$ sheaves and thus it is injective (since it is not zero).  So to show it is an isomorphism, it is sufficient to show it is surjective in codimension one.  Therefore we may consider the statement at the generic point $\gamma$ of a codimension 1 subvariety (locally, this is localizing at a height one prime).  Since $X$ is Gorenstein in codimension one, we see that
\[
\left(\sHom_{\O_X}(F^{1}_* \O_X, \O_X)\right)_{\gamma}
\]
is a free rank-one $F^1_* \O_X$-module.  We fix a generator $T_1$ and set $T_n$ to be the generator of
\[
\left(\sHom_{\O_X}(F^{n}_* \O_X, \O_X)\right)_{\gamma}
\]
obtained by composing $T_1$ with itself $(n-1)$-times just as in Equation \ref{CompositionEquation} ($T_n$ is a generator by Lemma \ref{LemmaCompositionGivesGens}).

If $\Delta$ does not contain the point $\gamma$ in its support, we are done by the previous lemma.  On the other hand, if $\Delta$ contains $\gamma$ in its support, then we may express $\Delta$ at the stalk of $\eta$ locally as $z^t$ (where $t$ is a rational number with denominator a factor of $p^e - 1$).  Then we notice that
\[
\begin{split}
T_e( z^{(p^e - 1)t} (F^d_* T_d(z_i^{\lceil (p^d - 1) t\rceil} \blank))) = T_e(F^d_* T_d (z^{\lceil p^d(p^e - 1)t + (p^d - 1)t\rceil} \blank) \\
= T_e (F^d_* T_d( z^{\lceil (p^{d+e} - 1)t \rceil} \blank)) = T_{e+d}(z^{\lceil (p^{d+e} - 1)t \rceil} \blank).
\end{split}
\]
This proves the corollary since for any $n > 0$, $T_n(z^{\lceil (p^n - 1)t \rceil} \blank)$ generates the image of the $F^n_* \O_{X, \gamma}$-module $\left(\sHom_{\O_X}(F^n_* \O_X(\lceil (p^n - 1)\Delta \rceil), \O_X) \right)_{\gamma}$ inside $\left( \sHom_{\O_X}(F^n_* \O_X, \O_X)\right)_{\gamma}$.

\end{proof}

We are now ready to explicitly relate $\phi : F^e_* \O_X \rightarrow \O_X$ to a $\bQ$-divisor $\Delta$.  As mentioned before, parts of this theorem were likely known to experts, but to my knowledge, it has not been written down in the language of $\bQ$-divisors.

\begin{theorem}
\label{TheoremMapsInduceDivisors}
Suppose that $R$ is a normal $F$-finite ring.
For every map $\phi : F^e_* R \rightarrow R$, there exists an effective $\bQ$-divisor $\Delta = \Delta_{\phi}$ on $X = \Spec R$ such that:
\begin{itemize}
\item[(a)]  $(p^e - 1) \Delta$ is an integral divisor.
\item[(b)]  $(p^e - 1)(K_X + \Delta)$ is a Cartier divisor and $\Hom_{R}(F^e_* R( (p^e - 1) \Delta), R) \cong F^e_* R$.
\item[(c)]  The natural map $F^e_* R \cong \Hom_{R}(F^e_* R( (p^e - 1) \Delta), R) \rightarrow \Hom_{R}(F^e_* R, R)$ sends some $F^e_* R$-module generator of $\Hom_{R}(F^e_* R( (p^e - 1) \Delta), R)$ to $\phi$.
\item[(d)]  The map $\phi$ is surjective if and only if the pair $(R, \Delta)$ is sharply $F$-pure.
\item[(e)]  The composition map
\[
\phi_{(n+1)e} = \phi \circ F^{e}_* \phi \circ F^{2e}_* \phi \circ \ldots \circ F^{ne}_* \phi
\]
also determines the same divisor $\Delta$.
\item[(f)]  Another map $\phi' : F^{e'}_* R \rightarrow R$ determines the same $\bQ$-divisor $\Delta$ if and only if for some positive integers $n$ and $n'$ such that $(n+1)e = (n' + 1) e'$ (equivalently, for every such pair of integers) there exists a unit $u \in R$ so that
    \[
        \phi \circ F^{e}_* \phi \circ F^{2e}_* \phi \circ \ldots \circ F^{ne}_* \phi (u x) =
         \phi' \circ F^{e'}_* \phi' \circ F^{2e'}_* \phi' \circ \ldots \circ F^{n'e'}_* \phi'(x).
    \]
    for all $x \in R$.
    In other words, $\phi$ and $\phi'$ determine the same divisor if and only if $\phi$ composed with itself $n$-times is a unit multiple of $\phi'$ composed with itself $n'$-times.
\end{itemize}
\end{theorem}
\begin{proof}
A map $\phi : F^e_* R \rightarrow R$ uniquely determines the map of $F^e_* R$-modules $\Phi : F^e_* R \rightarrow  \Hom_R(F^e_* R, R)$ which sends $1$ to $\phi$.  This can also be viewed as applying the functor $\Hom_R(\blank, R)$ to $\phi$ and factoring the map
\begin{equation}
\label{EquationFactoring}
\xymatrix{
R \ar[r]^-{\sim} \ar[dr]_-{F^e} & \Hom_R(R, R) \ar[r]^-{\phi^{\vee}} &\Hom_R(F^e_* R, R) \\
& F^e_* R \ar@{.>}[ur]_{1 \mapsto \phi} &\\
}
\end{equation}
through $F^e_* R$.  We know that $\Hom_R(F^e_* R, R) \cong F^e_* R((1-p^e)K_X + M)$ for some Cartier divisor $M$ (in many cases $M$ is zero, see Remark \ref{RemHaveToBeCareful}).  Therefore, the map $\Phi$ determines an effective divisor $D$ which is linearly equivalent to $(1 - p^e) K_X + M$, see \cite{Hartshorne} and Proposition \ref{PropSectionsAreSameAsEmbeddings}. Set
\[
\Delta := {1 \over p^e - 1}D.
\]
Clearly property (a) is satisfied.  For the first part of (b), simply note that
$(p^e - 1)(K_X + \Delta) = (p^e - 1)K_X + D \sim (p^e - 1)K_X + (1-p^e)K_X + M = M$.
For the second part of (b), observe that
\[
\begin{split}
\Hom_{R}(F^e_* R( (p^e - 1) \Delta), R) \cong F^e_* R((1-p^e)K_X + M - (p^e - 1) \Delta) \\
\cong F^e_* R((1-p^e)K_X + M  - D) \cong F^e_* R.
\end{split}
\]

Let us now prove (c).  At height one primes $\gamma$, the map $\Phi : F^e_* R_{\gamma} \rightarrow \Hom_R(F^e_* R, R)_{\gamma} \simeq F^e_* R_{\gamma}$ as above, is multiplication (as an $F^e_* R$-module) by a generator of $D$.  But so is the map from (c), $\Psi : F^e_* R \cong \Hom_{R}(F^e_* R( (p^e - 1) \Delta), R) \rightarrow \Hom_{R}(F^e_* R, R)$.  Note that all the modules involved are rank 1 reflexive $F^e_* \O_X$-modules and that the domains of $\Phi$ and $\Psi$ are isomorphic.  Therefore the maps $\Phi$ and $\Psi$ induce the same divisors and so $\Phi$ and $\Psi$ can be identified (for an appropriate choice of isomorphism $\Hom_{R}(F^e_* R( (p^e - 1) \Delta), R) \cong F^e_* R$).  Part (c) then follows.

To prove (d), suppose first that $\phi$ is surjective, or equivalently that $1$ is in $\phi$'s image.  Then there exists an $R$-module map $\alpha$ so that the composition $\xymatrix{ R \ar[r]^-{\alpha} & F^e_* R \ar[r]^-{\phi} & R}$ is the identity.  Apply $\Hom_R(\blank, R)$ to the diagram \ref{EquationFactoring}.  This gives a diagram:
\[
\xymatrix{
R \ar@{<-}[rr]^-{\phi} \ar@{<-}[dr] &  & \Hom_R(\Hom_R(F^e_* R, R), R) \ar@{<->}[r]^-{\sim} & F^e_* R \\
& \Hom_R(F^e_* R, R) \ar@{<->}[d]^-{\sim} \ar@{<.}[ur] &\\
& F^e_* R(D)
}
\]
and so we can factor $\phi$ as $F^e_* R \rightarrow F^e_* R(D) \rightarrow R$. This proves that $(R, \Delta)$ is a sharply $F$-pure pair.  Conversely, suppose that $(R, \Delta)$ is sharply $F$-pure, then a single (equivalently every) generator $\alpha$ of $\Hom_{R}(F^e_* R( (p^e - 1) \Delta), R)$ satisfies $\alpha(F^e_* R) = R$.  But $\phi$ is such a generator so $\phi(F^e_* R) = R$.

We now prove (e).  It is enough to check the statement at a height one prime $\gamma$.  We know that $\Hom_R(F^e_* R, R)_{\gamma}$ is locally free of rank one with generator $T_e$.  We then see that $\phi_{\gamma}(\blank) = T_e(d \blank)$ where $d$ is a defining equation for $D$ when localized at $\gamma$.  Composing this with itself $n$-times, we obtain the map
\[
\phi_{\gamma} \circ F^{e}_* \phi_{\gamma} \circ F^{2e}_* \phi_{\gamma} \circ \ldots \circ F^{ne}_* \phi_{\gamma} (F^{(n+1)e}_* z) = T_{(n+1)e}(F^{(n+1)e}_* d^{p^{ne} + p^{(n-1)e} + \dots + p^e + 1} z).
\]
But now we notice that ${1 \over p^{(n+1)e} - 1} (p^{ne} + p^{(n-1)e} + \dots + p^e + 1) D$ is equal to ${1 \over p^e - 1} D$.

Finally, we prove (f).  First note that changing a map by pre-composing with multiplication by a unit does not change the associated divisor.  Therefore, if maps $\phi$ and $\phi'$ satisfy the condition on their compositions (as above), then they determine the same divisor by (e).  Conversely, suppose that the maps $\phi$ and $\phi'$ have the same associated divisor and choose $n$ and $n'$ as above.  Without loss of generality, by replacing $\phi$ and $\phi'$ with their compositions, we may assume that $e = e'$, and we simply have two maps $\phi, \phi' \in \Hom_R(F^e_* R, R)$ that determine the same divisor.  In particular, the maps
\[
\xymatrix@R=1pt{
F^e_* R \ar[r] & \Hom_R(F^e_* R, R) & & F^e_* R \ar[r] & \Hom_R(F^e_* R, R) \\
 1 \ar@{|->}[r] & \phi & &  1 \ar@{|->}[r] & \phi' \\
}
\]
induce the same embedding of $\Hom_R(F^e_* R, R)$ into the total field of fractions of $F^e_* R$.  Therefore the two maps differ by multiplication by a unit as desired, see \cite{HartshonreGeneralizedDivisorsAndBiliaison} or Proposition \ref{PropSectionsAreSameAsEmbeddings}.
\end{proof}

\begin{remark}
Note that condition (a) above is redundant in view of condition (b).
\end{remark}

\begin{theorem}
\label{TheoremDivisorsInduceMaps}
Suppose that $R$ is normal and $F$-finite as above.
For every effective $\bQ$-divisor $\Delta$ satisfying conditions (a) and (b) from Theorem \ref{TheoremMapsInduceDivisors}, there exists a map $\phi \in \Hom_R(F^e_* R, R)$ such that the divisor associated to $\phi$ is $\Delta$.
\end{theorem}
\begin{proof}
We set $\phi$ to be the image of $1$ under the composition
\[
i \circ q \circ F^e : R \rightarrow F^e_* R \cong \Hom_R(F^e_* R((p^e - 1) \Delta), R) \rightarrow \Hom_R(F^e_* R, R),
\]
where $q$ is the isomorphism given by hypothesis, and $i$ the map induced by the inclusion $F^e_* R \subseteq F^e_* R((p^e - 1)\Delta)$.
It is straightforward to verify that applying $\Hom_R(\blank, R)$ to the above composition also explicitly constructs (and factors) $\phi$ because of the isomorphism $\Hom_R(\Hom_R(F^e_* R, R), R) \cong F^e_* R$.

Applying $\Hom_R(\blank, R)$ to this factorization of $\phi$, and using the construction from Theorem \ref{TheoremMapsInduceDivisors} gives us back $\Delta$.
\end{proof}

In summary, we have shown that for a reduced normal $F$-finite local ring $R$ there is a bijection between the sets
\[
\left\{ \begin{matrix}\text{Effective $\bQ$-divisors $\Delta$ }\\\text{such that $(p^e - 1)(K_X + \Delta)$ is Cartier}\end{matrix} \right\} \leftrightarrow \left\{ \text{Non-zero elements of $\Hom_R(F^e_* R, R)$} \right\} \Big/ \sim
\]
where the equivalence relation on the right identifies two maps $\phi$ and $\psi$ if there is a unit $u \in R$ such that $\phi(u \times \blank) = \psi(\blank)$.  There will be some discussion of how to make sense of such a correspondence in the non-local case in Remark \ref{RemarkGlobalGluing}.

One can even extend this correspondence further.  Recall that putting an $R\{F^e\}$-module structure on an $R$-module $M$ is equivalent to specifying an additive map
\[
 \phi_e : M \rightarrow M
\]
such that $\phi_e(rm) = r^{p^e} \phi_e(m)$; see \cite{LyubeznikSmithCommutationOfTestIdealWithLocalization} for additional details.  Such maps can also be identified with $R$-module maps $M \rightarrow F^e_* M$.

\begin{proposition}
Suppose that $(R, \bm)$ is a complete normal local $F$-finite ring with injective hull of the residue field $E_R$.   Then there is a bijection between the set of $R\{F^e\}$-module structures on $E_R$ and the set of elements of $\Hom_R(F^e_* R, R)$.
\end{proposition}
\begin{proof}
Consider a map $\phi : F^e_* R \rightarrow R$ and apply $\Hom_R(\blank, E_R)$.  This gives us a map
\[
E_R = \Hom_R(R, E_R) \rightarrow \Hom_R(F^e_* R, E_R) = E_{F^e_* R} = F^e_* E_R.
\]
Applying $\Hom_R(\blank, E_R)$ gives us back $\phi$.  Note that there are (non-canonical) choices here when we identify $F^e_* E_R$ with $\Hom_R(F^e_* R, E_R)$.  However, these are merely up to multiplication by units and so we can fix such isomorphisms.
\end{proof}

Therefore, in the case of a complete local normal ring, we have the following correspondence.
\[
\begin{split}
\left\{ \begin{matrix}\text{Effective $\bQ$-divisors $\Delta$}\\ \text{such that $(p^e - 1)(K_X + \Delta)$} \\ \text{is Cartier} \end{matrix} \right\} \leftrightarrow \left\{ \begin{matrix} \text{Nontrivial cyclic $F^e_* R$-submodules} \\ \text{of $\Hom_R(F^e_* R, R)$} \end{matrix} \right\} \\ \leftrightarrow \left\{ \begin{matrix}\text{Nonzero elements of} \\ \text{$\Hom_R(F^e_* R, R)$}  \end{matrix} \right\} \Big/ \sim \leftrightarrow \left\{ \begin{matrix} \text{Nonzero $R\{F^e\}$-module} \\ \text{structures on $E_R$} \end{matrix} \right\} \Big/ \sim
\end{split}
\]
The first equivalence relation identified two maps if they agree up to pre-composition with multiplication by a unit of $F^e_* R$ (as above).  The second equivalence relation identified two maps if they agree up to post-composition with multiplication by a unit of $F^e_* R$.

\begin{corollary}
Suppose that $S$ is a regular $F$-finite ring such that $F^e_* S$ is free as an $S$-module and that $R = S/I$ is a quotient that is normal.  Further suppose that $\Hom_R(F^e_* R, R) \cong F^e_* R$ (in particular, $R$ is $\bQ$-Gorenstein with index not divisible by $p$).  Write $(I^{[p^e]} : I) = I^{[p^e]} + (f_e)$ just as in Corollary \ref{CorDescriptionOfFedderColon}.  Then for all $n > 0$,
\[
(I^{[p^{ne}]} : I) = I^{[p^{ne}]} + (f_e^{1 + p^e + \dots + p^{(n-1)e}}).
\]
\end{corollary}


\section{Application to centers of sharp $F$-purity}
\label{SectionApplicationsToCenters}

In \cite{SchwedeCentersOfFPurity}, we introduced a notion called centers of sharp $F$-purity (also known as $F$-compatible ideals), a positive characteristic analog of a center of log canonicity; see for example \cite{KawamataOnFujitasFreenessConjectureFor3Folds} and \cite{KawamataSubadjunction2}.
Our main goal in this section is to prove several finiteness theorems about centers of sharp $F$-purity.

Recall that an ideal $I$ is called $F$-compatible with respect to $(R, \Delta)$ if for every $e > 0$ and every $\phi \in \Hom_R(F^e_* R(\lceil (p^e - 1) \Delta \rceil), \Delta)$, we have $\phi(F^e_* I) \subseteq I$.  One limitation of the definition of $F$-compatible ideals is that it seems to require checking infinitely many $e > 0$ (and infinitely many $\phi$).  However, for radical ideals $I$, assuming that $(p^e - 1)K_X$ is Cartier, we will show that it is enough to check only that $e$.

\begin{proposition}
\label{PropUniformlyFCompatIfAndOnlyIf}
Suppose that $R$ is a normal $F$-finite ring.  Further suppose that $\Delta$ is an effective $\bQ$-divisor such that $\Hom_R(F^e_* R((p^e - 1)\Delta), R)$ is free as an $F^e_* \O_X$-module.  Then a radical ideal $I \subset R$ is $F$-compatible with respect to $(R, \Delta)$ if and only if $T_e(F^e_* I) \subseteq I$ where $T_e$ is a $F^e_*R$-module generator of $\Hom( F^e_* R( (p^e - 1) \Delta), R)$.
\end{proposition}
\begin{proof}
Since a radical ideal $I$ is $F$-compatible if and only if its minimal associated primes are $F$-compatible, see Lemma \ref{LemmaPropertiesOfUniformlyFCompatible}(5), without loss of generality we may assume that $I$ is prime.  Furthermore, since $F$-compatible ideals behave well with respect to localization, see \cite[Lemma 3.7]{SchwedeCentersOfFPurity}, we may also assume that $R$ is local and that $I = \bm$ is maximal.

Suppose that $\phi : F^b_* R (\lceil (p^b - 1) \Delta\rceil) \rightarrow R$ satisfies the property that $\phi(F^b_* \bm) \nsubseteq \bm$, we will obtain a contradiction.  Therefore, for some element $x \in \bm$, we have that $\phi(F^b_* x) = u$ where $u$ is a unit in $R$.
By scaling $\phi$, we may assume that $u = 1$.  Now choose integers $n$ and $m$ such that $nb = me$.  Consider the function $\psi : F^{nb}_* R \rightarrow R$ defined by the rule
\[
 \psi(F^{nb}_* \blank) = \phi(xF^b_* \phi(x F^{2b}_* \phi(x \cdots F^{(n-1)b}_* \phi(F^{nb}_* \blank) \cdots ))).
\]
Notice that $\psi(F^{nb}_* x) = 1$.  On the other hand, $\Hom_R(F^{me}_* R( (p^{me} - 1)\Delta), R)$ is generated by $T$ composed with itself $(m-1)$-times.  Notice that since $T$ sends $\bm$ into $\bm$, so does its composition.  Therefore, to obtain our contradiction we simply have to check that $\psi \in \Hom_R(F^{nb}_* R, R)$ is an element of $\Hom_R(F^{me}_* R( (p^{me} - 1)\Delta), R)$.  But that is straightforward since it was constructed by composing $\phi$ with itself (using the fact that we round-up, and not round-down, so that $p^a\lceil (p^b - 1)\Delta \rceil + \lceil (p^a - 1) \Delta \rceil \geq \lceil (p^{a+b} - 1) \Delta \rceil$ ).
\end{proof}

\begin{remark}
 For a sharply $F$-pure pair $(R, \Delta)$, all $F$-compatible ideals are radical.
\end{remark}

\begin{corollary}
\label{CorCompatFrobeniusSplitIdealsAreCenters}
Suppose that $\phi : F^e_* R \rightarrow R$ is a Frobenius splitting and $R$ is an $F$-finite normal ring.  Then the centers of sharp $F$-purity for the pair $(R, \Delta_{\phi})$ coincide with the subschemes of $X = \Spec R$ compatibly split with $\phi$.
\end{corollary}

\begin{remark}
\label{RemarkQuestionCheckOneNonRadical}
One might ask if an analog of Proposition \ref{PropUniformlyFCompatIfAndOnlyIf} holds for non-radical ideals, and we do not know the answer in general.  However, in \cite{SchwedeCentersOfFPurity}, it was shown that the non-finitistic/big test ideal is the unique smallest $F$-compatible ideal that intersects non-trivially with $R^{\circ}$.  Using the additional structure of the big test ideal, we are able to prove an analogous result (in fact, the proof is very similar to a special case of \cite[Proposition 3.5(3)]{TakagiPLTAdjoint}).
\end{remark}

\begin{definition}
\label{DefinitionComposingMaps}
Suppose that $\phi_e \in \Hom_R(F^e_* R, R)$ is a map.  For every integer $n \geq 0$, we define $\phi_{ne} \in \Hom_R(F^{ne}_* R, R)$ to be the map obtained by composing $\phi_e$ with itself $(n-1)$-times, just as in Theorem \ref{TheoremMapsInduceDivisors}(e).  We set $\phi_{0}$ to be the identity map in $\Hom_R(R, R)$.
\end{definition}

Our next goal is to characterize the big test ideal using this machinery.  First however, we need two lemmas.

\begin{lemma}
\label{LemmaFillInGaps}
Suppose that $\ba$ is an ideal generated by $l$ elements and that $m$ and $k$ are integers.  Then:
\[
(\ba^m)^{[p^k]} \supseteq \ba^{p^k m + l(p^k - 1)}.
\]
\end{lemma}
\begin{proof}
Let $f_1, \ldots, f_l$ be a set of generators for $\ba$.  Then $\ba^{p^k m + l(p^k - 1)}$ is generated by the elements of the form
\[
f_1^{b_1} \dots  f_l^{b_l}
\]
where $\sum_{i = 1}^l b_i = p^k m + l(p^k - 1)$.  We will show that each such element is contained in $(\ba^m)^{[p^k]}$. Write each $b_i = q_i p^k + r_i$ where $0 \leq r_i < p^k$.  Thus we have
\[
f_1^{b_1} \cdots  f_l^{b_l} = (f_1^{q_1} \cdots f_l^{q_l})^{p^k} (f_1^{r_1} \cdots f_l^{r_l}).
\]
Note $\sum_{i = 1}^{l} r_i \leq l(p^k - 1)$.  Therefore,
\[
p^k m + l(p^k - 1) = \sum_{i=1}^l b_i  = \left(p^k \sum_{i = 1}^l q_i\right)  + \left(\sum_{i = 1}^l r_i\right) \leq \left(p^k \sum_{i = 1}^l q_i\right) + l(p^k - 1)
\]
which implies that $p^k m \leq p^k \sum_{i = 1}^l q_i$, in particular, $m \leq \sum_{i = 1}^l q_i$.  Therefore,
\[
(f_1^{q_1} \cdots f_l^{q_l})^{p^k} \in (\ba^m)^{[p^k]}
\]
and so $f_1^{b_1} \cdots  f_l^{b_l} \in (\ba^m)^{[p^k]}$ as desired.
\end{proof}

\begin{lemma}
 \label{LemUniformGrowthOfFrobeniusPowers}
Suppose that $\ba$ is an ideal of $R$ which can be generated by $l$ elements and such that $\ba \cap R^{\circ} \neq \emptyset$.  Fix an $e > 0$.  Then there exists an element $c' \in R^{\circ}$ such that
\[
 c' \ba^{\lceil t(p^{ne + k} - 1) \rceil} \subseteq (\ba^{\lceil t(p^{ne} - 1) \rceil})^{[p^k]}
\]
for all $n > 0$ and all $k < e$.
\end{lemma}
\begin{proof}
First note that we have 
\[
(\ba^{\lceil t(p^{ne} - 1) \rceil})^{[p^k]} \supseteq \ba^{p^k \lceil t(p^{ne} - 1) \rceil + l (p^k - 1)} \supseteq \ba^{p^k \lceil t(p^{ne} - 1) \rceil +  l p^e}.
\]
The first containment holds by Lemma \ref{LemmaFillInGaps} above.  Thus it is sufficient to find a $c'$ such that $c' \ba^{\lceil t(p^{ne + k} - 1) \rceil} \subseteq \ba^{p^k \lceil t(p^{ne} - 1) \rceil + l p^e}$.  Choose $c' \in \ba^{(l+1)p^e} \cap R^{\circ}$.  We need to show that
\[
(l+1)p^e + \lceil t(p^{ne + k} - 1) \rceil \geq p^k \lceil t(p^{ne} - 1) \rceil + l p^e.
\]
However,
\[
p^k \lceil t(p^{ne} - 1) \rceil + l p^e \leq p^k \lfloor t(p^{ne} - 1) \rfloor + p^e + l p^e \leq \lfloor p^k t(p^{ne} - 1) \rfloor + (l+1) p^e \leq \lceil t(p^{ne + k} - 1) \rceil + (l+1)p^e
\]
as desired.

\end{proof}

\begin{proposition}
\label{PropositionBigTestIdealIsSmallest}
Suppose that $R$ is a normal $F$-finite ring.  Further suppose that $\Delta$ is an effective $\bQ$-divisor such that $(p^e - 1)\Delta$ is integral and that $\Hom_R(F^e_* R((p^e - 1)\Delta), R)$ is rank one and free as an $F^e_* R$-module with generator $T_e$ (viewed as an element of $\Hom_R(F^e_* R, R)$).  Set $T_{ne}$ to be the map obtained by composing $T_{e}$ with itself $(n-1)$-times.  Then we have the following
\begin{itemize}
\item[(i)]  The big test ideal $\tau_b(R; \Delta)$ is the unique smallest ideal $J$ whose intersection with $R^{\circ}$ is non-trivial and which satisfies $T_e(F^e_* J) \subseteq J$.
\item[(ii)]  Furthermore, if $\ba$ is an ideal such that $\ba \cap R^{\circ} \neq \emptyset$ and $t > 0$ is a real number, then the big test ideal $\tau_b(R; \Delta, \ba^t)$ is the unique smallest ideal $J$ whose intersection with $R^{\circ}$ is non-trivial and which satisfies $T_{ne}(F^{ne}_* \ba^{\lceil t(p^e - 1) \rceil} J) \subseteq J$ for all integers $n > 0$.
\end{itemize}
\end{proposition}
\begin{proof}  
For (i), note that the big test ideal $\tau_b(R; \Delta)$ satisfies the condition $T_e(F^e_* \tau_b(R; \Delta)) \subseteq \tau_b(R; \Delta)$ due to \cite[Proposition 6.1]{SchwedeCentersOfFPurity}.  Thus we simply have to show it is the smallest such ideal.  Likewise for (ii), $\tau_b(R; \Delta, \ba^t)$ satisfies the condition $T_{ne}(F^{ne}_* \ba^{\lceil t(p^e - 1) \rceil} \tau_b(R; \Delta, \ba^t)) \subseteq \tau_b(R; \Delta, \ba^t)$ for all integers $n > 0$, so we must show that it is the smallest such ideal.

We now claim that statement is local in order to assume that $R = (R, \bm)$ is a local ring.  We outline the proof of this claim in case (i) since case (ii) is essentially the same.
Suppose that $J$ is an ideal which satisfies both $J \cap R^{\circ} \neq \emptyset$ and $T_e(F^e_* J) \subseteq J$.  Then $J + \tau_b(R; \Delta)$ also satisfies both conditions.  Note that $J$ does not contain $\tau_b(R; \Delta)$ if and only if we have the strict containment $J + \tau_b(R; \Delta) \supsetneq J$.  But in such a case, we can localize at a maximal ideal where the same strict containment holds.  Thus we have reduced to the local case.  Therefore, from this point forward, we assume that $R$ is a local ring with maximal ideal $\bm$.

Suppose that $J$ is an ideal such that $T_e(F^e_* J) \subseteq J$ (respectively $T_{ne}(F^{ne}_* \ba^{\lceil t(p^{ne} - 1) \rceil} J) \subseteq J$ for all $n > 0$) and such that $J \cap R^{\circ} \neq \emptyset$.    In case (i), notice also that $T_{ne}(F^{ne}_* J) \subseteq J$ for all positive integers $n$ (and thus $\phi(F^{ne}_* J) \subseteq J$ for all $\phi \in \Hom_R(F^{ne}_* R( (p^{ne - 1}-1)\Delta), R)$ since $T_{ne}$ is also a generator by Corollary \ref{CorEveryMapComposesWithAGenerator}).

In the setting of (i), fix $d \in J \cap R^{\circ}$. 
By applying Matlis duality, we see that the composition
\[
\xymatrix@C=10pt{
 E_{R/J} \ar[r] & E_R \ar[r] & E_R \tensor_R F^{ne}_* R \ar[r] & E_R \tensor_R F^{ne}_* R((p^{ne} - 1)\Delta) \ar[rr]^{F^{ne}_* (\times d)} & & E_R \tensor_R F^{ne}_* R((p^{ne} - 1)\Delta)
}
\]
is zero for every integer $n > 0$.  Likewise, in the setting of (ii), for each $d \in J \cap R^{\circ}$ and each $a \in \ba^{\lceil t(p^{ne} - 1) \rceil}$, we have that the composition
\[
\xymatrix@C=10pt{
 E_{R/J} \ar[r] & E_R \ar[r] & E_R \tensor_R F^{ne}_* R \ar[r] & E_R \tensor_R F^{ne}_* R((p^{ne} - 1)\Delta) \ar[rr]^{F^{ne}_* (\times d a)} & & E_R \tensor_R F^{ne}_* R((p^{ne} - 1)\Delta)
}
\]
is zero for every integer $n > 0$.

We now want to show that $E_{R/J} \subset 0^{*\Delta}_{E_R}$ (respectively $E_{R/J} \subset 0^{* \Delta, \ba^t}_{E_R}$) because $\Ann_R(0^{*\Delta}_{E_R}) = \tau_b(R; \Delta)$ (respectively $\Ann_R(0^{* \Delta, \ba^t}_{E_R}) = \tau_b(R; \Delta, \ba^t)$).
Therefore, choose $z \in E_{R/J}$.  By assumption $d z^{p^{ne}} = 0 \in E_R \tensor_R F^{ne}_* R((p^{ne} - 1)\Delta)$ for all $n > 0$ (respectively, $d \ba^{\lceil t(p^{ne} - 1) \rceil} z^{p^{ne}} = 0 \in E_R \tensor_R F^{ne}_* R((p^{ne} - 1)\Delta)$  for all $n > 0$).  We need to verify a similar statement for powers of $p$ that are not multiples of $e$, and so now the proof becomes quite similar to \cite[Lemma 8.16]{HochsterHunekeTC1}.

In the setting of (i), we claim that $F^{ne}_* R((p^{ne} - 1)\Delta )$ naturally maps to $F^{k+ne}_* R(\lceil (p^{k+ne} - 1)\Delta \rceil)$ for any $k > 0$ via the $k$-iterated action of Frobenius.  To see this explicitly, apply $\Hom_R(R(-\lceil (p^{ne} - 1) \Delta\rceil), \blank)$ to the map $R \rightarrow F^k_* R(\lceil (p^{k} - 1)\Delta \rceil)$.  Tensoring with $E_R$ then gives us a map
\[
\xymatrix@R=3pt{
F^{ne}_* R((p^{ne} - 1)\Delta ) \tensor_R E_R \ar[r] & F^{k+ne}_* R(\lceil (p^{k+ne} - 1)\Delta \rceil) \tensor_R E_R \\
d z^{p^ne} = d \tensor z \ar@{|->}[r] & d^{p^k} \tensor z = d^{p^k} z^{p^{k + ne}}
}
\]
which factors the map $E_R \rightarrow F^{k+ne}_* R(\lceil (p^{k+ne} - 1)\Delta \rceil) \tensor_R E_R$.
Therefore, $d^{p^k} z^{p^{ne + k}} = 0$ for all $k, n > 0$.

Choose $c = d^{p^{e - 1}}$ and choose $j > 0$ arbitrary.  Write $j = ne + k$ where $k < e$.  Then
\[
c z^{p^j} = d^{p^{e - 1}} z^{p^{ne + k}} = d^{p^{e - 1} - p^k} d^{p^k} z^{p^{ne + k}} = d^{p^{e - 1} - p^k} 0 = 0
\]
as desired.  Therefore, $E_{R/J} \subset 0^{*\Delta}_{E_R}$ so that $J = \Ann_R(E_{R/J}) \supseteq \Ann_R( 0^{*\Delta}_{E_R}) = \tau_b(R; \Delta)$ which proves (i).

In case (ii), using a similar argument, we still have that $d^{p^k} (\ba^{\lceil t(p^{ne} - 1) \rceil})^{[p^k]} z^{p^{ne + k}} = 0$ for all $k, n > 0$.  By Lemma \ref{LemUniformGrowthOfFrobeniusPowers}, there exists a $c' \in R^{\circ}$ such that $c' \ba^{\lceil t(p^{ne + k} - 1) \rceil} \subseteq (\ba^{\lceil t(p^{ne} - 1) \rceil})^{[p^k]}$
for all $n > 0$ and all $k < e$.

Set $c = c' d^{p^{e - 1}}$, choose $j > 0$ arbitrary and write $j = ne + k$ where $k < e$.  Then
\[
c \ba^{\lceil t(p^j - 1)\rceil} z^{p^j} = d^{p^{e - 1}} c' \ba^{\lceil t(p^{ne + k} - 1)\rceil} z^{p^{ne + k}} \subseteq d^{p^{e - 1} - p^k} d^{p^k} (\ba^{\lceil t(p^{ne} - 1) \rceil})^{[p^k]} z^{p^{ne + k}} = d^{p^{e - 1} - p^k} 0 = 0
\]
as desired.
\end{proof}

\section{$F$-adjunction}

In this section, we re-interpret the following observation using the language from the previous sections.

\begin{observation}
\label{ObsMapsRestrictToCenters}
Suppose that $(R, \bm)$ is an $F$-finite local ring and $\phi \in \Hom_R(F^e_* R, R)$.  Further suppose that $I$ is a proper ideal of $R$ such that $\phi(F^e_* I) \subseteq I$.  Then there is the following diagram:
\[
\xymatrix{
F^e_* R \ar[d]_{F^e_* \alpha} \ar[r]^-{\phi} & R \ar[d]^{\alpha} \\
F^e_* (R/I) \ar[r]^-{\phi_I} & R/I \\
}
\]
where the vertical arrows are the natural quotients.

\begin{itemize}
\item{} Because $R$ is local, $\phi$ is surjective if and only if $\phi_I$ is surjective.
\end{itemize}
\end{observation}

When we apply the correspondence between effective $\bQ$-divisors and $\phi \in \Hom_R(F^e_* R, R)$, we obtain the following result.

\begin{theorem}
\label{ThmFirstFAdjunction}
Suppose that $R$ is a reduced $F$-finite normal ring and that $(R, \Delta)$.  Assume also that $(p^e - 1)\Delta$ is an integral divisor such that we have an isomorphism $\Hom_R(F^e_* R((p^e - 1)\Delta), R) \cong F^e_*R$ of $F^e_*R$-modules.  Further suppose that $I \subset R$ is $F$-compatible with respect to $(R, \Delta)$ and that $R/I$ is normal. Finally suppose that $(R, \Delta)$ is sharply $F$-pure at the generic points of $\Spec R/I$ (that is, after localizing at the minimal primes of $I$).  Then there exists a canonically determined effective $\bQ$-divisor $\Delta_{R/I}$ on $\Spec R/I$ satisfying the following properties:
\begin{itemize}
\item[(i)]  $(p^e - 1) (K_{R/I} + \Delta_{R/I})$ is an integral Cartier divisor
\item[(ii)]  $\Hom_{R/I}\left(F^e_* \left((R/I)((p^e - 1)\Delta_{R/I})\right), R/I\right) \cong F^e_*(R/I)$ as $F^e_*(R/I)$-modules.
\item[(iii)]  $(R, \Delta)$ is sharply $F$-pure near $\Spec R/I$ if and only if $(R/I, \Delta_{R/I})$ is sharply $F$-pure.
\item[(iv)]  For any ideal $\ba \subseteq R$ which is not contained in any minimal prime of $I$ and any real number $t > 0$, we have that
    $(R, \Delta, \ba^t)$ is sharply $F$-pure near $\Spec R/I$ if and only if $(R/I, \Delta_{R/I}, {\overline \ba}^t)$ is sharply $F$-pure.
\item[(v)]  $I$ is maximal with respect to containment among $F$-compatible ideals for the pair $(R, \Delta)$ (in other words, $I$ is a minimal center of sharp $F$-purity), if and only if $(R/I, \Delta_{R/I})$ is a strongly $F$-regular pair and $R/I$ is a domain.\footnote{In fact, if we we assume that $I$ is maximal among $F$-compatible ideals, then it follows that $R/I$ is a normal domain and so the assumption that $R/I$ is normal is unnecessary.}
\item[(vi)]  There is a natural bijection between the centers of sharp $F$-purity of $(R/I, \Delta_{R/I})$, and the centers of sharp $F$-purity of $(R, \Delta)$ which contain $I$.
\end{itemize}
\end{theorem}

\begin{remark}
Roughly speaking, properties (iii), (iv), (v) and (vi) imply is that the singularities of $(R/I, \Delta_{R/I})$ are very closely related to the singularities of $(R, \Delta)$ near $I$.  Compare with \cite{KawakitaInversion}, \cite{KawamataSubadjunction2}, \cite{KawakitaComparisonNonLCI},  \cite{EinMustataJetSchemesAndSings}, \cite{AmbroAdjunctionConjecture}, and \cite{EinMustataYasuda}.
\end{remark}

\begin{proof}
Given $\Delta$ as above, associate a $\phi \in \Hom_R(F^e_* R, R)$ as in Theorem \ref{TheoremDivisorsInduceMaps}.  Just as in Observation \ref{ObsMapsRestrictToCenters}, we associate a $\phi_I \in \Hom_{R/I}(F^e_* (R/I), R/I)$, to which we associate a divisor $\Delta_{R/I}$.  By construction (and using Theorem \ref{TheoremMapsInduceDivisors}) we see that the existence and that properties (i) and (ii) are obvious.  For the rest of the properties, it is harmless to assume that $R$ is local.  Notice that the map $\phi_I$ is not the zero map on any irreducible component of $\Spec R/I$ because $(R, \Delta)$ is sharply $F$-pure at the minimal primes of $I$.  To show that $\Delta_{R/I}$ is canonically determined, note that if one chooses a different $\phi : F^e_ * R \rightarrow R$ associated to $\Delta$, the associated map $\phi_I$ will differ from the original choice by multiplication by a unit, and so $\Delta_{R/I}$ will not change.  Likewise, if one chooses a different $e > 0$, then using Theorem \ref{TheoremMapsInduceDivisors}(e,f), we obtain the same $\Delta_{R/I}$ yet again.

In terms of (iii), this simply follows from Observation \ref{ObsMapsRestrictToCenters}.  Notice now that (iv) is a generalization of (iii).  Condition (iv) follows by an argument similar to the one in Observation \ref{ObsMapsRestrictToCenters} since we simply consider a diagram
\[
\xymatrix@C=66pt{
F^d_* R \ar[d]^{F^d_* \alpha} \ar[d]^{F^d_* \alpha} \ar[r]^{F^d_* (\times a)} & F^d_* R \ar[d]^{F^d_* \alpha} \ar[r]^{\phi^n} & R \ar[d]^{\alpha} \\
F^d_* R/I \ar[r]^{F^d_* (\times \overline{a})} &  F^d_* R/I  \ar[r]^{\phi_I^{n}} & R/I \\
}
\]
for each $d = ne$ instead and various $a \in \ba^{\lceil t(p^d - 1) \rceil}$. In the diagram above, $\phi^n$ is the composition of $\phi$ with itself $(n-1)$-times as before.  Now again, the map obtained by composing the bottom row is surjective if an only if the map obtained from composing the top row is surjective.

Condition (v) will follow from (vi) since a pair is strongly $F$-regular if and only if it has no centers of sharp $F$-purity.  Therefore, we conclude by proving (vi). Suppose that $P \in \Spec R$ contains $I$, and corresponds to $\overline{P} \in \Spec R/I$.  We will show that $P$ is a center of sharp $F$-purity of $(R, \Delta)$ if and only if $\overline{P}$ is a center of sharp $F$-purity for $(R/I, \Delta_{R/I})$.  First suppose that $P$ is a center of sharp $F$-purity for $(R, \Delta)$.  This is equivalent to the condition that $\phi(F^e_* P) \subseteq P$.  This implies that $\phi_I(F^e_* \overline{P}) \subseteq \overline{P})$.  The converse direction reverses this and is essentially the same as the argument given in the proof of \cite[Proposition 7.5]{SchwedeCentersOfFPurity}.
\end{proof}

\begin{remark}
I do not know if one can somehow generalize the ``centers of sharp $F$-purity'' of condition (vi) to all $F$-compatible ideals.  It is not hard to see that one does obtain a bijection between radical $F$-compatible ideals since they are intersections of centers of sharp $F$-purity.  Section \ref{SectionAdjointLikeTestIdeals} is concerned with proving an analog of (vi) for the big test ideal.
\end{remark}

Using the ideas of Fedder's criterion, we also obtain the following result.

\begin{theorem}
\label{ThmMainAppOfFedder}
Suppose that $S$ is a regular $F$-finite ring such that $F^e_*S$ is a free $S$ module (for example, if $S$ is local) and that $R = S/I$ is a quotient that is a normal domain.  Further suppose that $\Delta_R$ is an effective $\bQ$-divisor on $\Spec R$ such that $\Hom_R(F^e_* R((p^e - 1) \Delta ), R)$ is a rank one free $F^e_* R$-module (for example, if $R$ is local and $(p^e - 1)(K_R + \Delta$) is Cartier).  Then there exists an effective $\bQ$-divisor $\Delta_S$ on $\Spec S$ such that:
\begin{itemize}
\item[(a)]  $(p^e - 1)(K_S + \Delta_S)$ is Cartier.
\item[(b)]  $I$ is $(\Delta_S, F)$-compatible and $(S, \Delta_S)$ is sharply $F$-pure at the minimal associated primes of $I$ (that is, at the generic points of $\Spec S/I$).
\item[(b)]  $\Delta_S$ induces $\Delta_R$ as in Theorem \ref{ThmFirstFAdjunction}.
\end{itemize}
\end{theorem}
\begin{proof}
The key point is that every map $F^e_* R \rightarrow R$ is obtained by restricting a map $F^e_* S \rightarrow S$ to $R$, see \cite[Lemma 1.6]{FedderFPureRat}.  Note that condition (b) follows immediately since the map $F^e_* R \rightarrow R$ we are concerned with is non-zero.
\end{proof}

\begin{remark}
The $\Delta_S$ constructed in the above theorem is in no way canonically chosen.
\end{remark}

\begin{remark}
I do not know of anything like a characteristic zero analog of this except in the case that $X \subseteq Y$ is a complete intersection, see \cite{EinMustata}, also compare with \cite{KawakitaComparisonNonLCI} and \cite{EinMustataJetSchemesAndSings}.
\end{remark}

We now show that for an $F$-pure pair, there are at most finitely many centers of sharp $F$-purity (equivalently there are at most finitely many $(\Delta,F)$-compatible ideals).  We give a proof that is written using the language of divisors.  However the same proof may be given without this language (this was done in a preprint of this paper).  This result was proved for local rings in \cite[Corollary 5.2]{SchwedeCentersOfFPurity}, using the method of \cite{EnescuHochsterTheFrobeniusStructureOfLocalCohomology} or a modification of the method of \cite{SharpGradedAnnihilatorsOfModulesOverTheFrobeniusSkewPolynomialRing}.  Finally, essentially the same result has also been obtained independently by Kumar and Mehta, \cite{KumarMehtaFiniteness}.

\begin{theorem}
\label{ThmFinitelyManyCenters}
 If $(R, \Delta, \ba_{\bullet})$ is sharply $F$-pure, then there are at most finitely many centers of sharp $F$-purity.
\end{theorem}
\begin{proof}
We may prove this on a finite affine cover of $\Spec R$.  Thus, we may assume\footnote{This happens after localizing each point, so it happens in a neighborhood of each point, so we may use such neighborhoods to cover $\Spec R$} there exists a map $\phi : F^e_* R(\lceil (p^e - 1) \Delta \rceil) \rightarrow R$ that sends some element $a \in F^e_* \ba_{p^e - 1}$ to $1$.  Note, every center of sharp $F$-purity $Q \in \Spec R$ for $(R, \Delta, \ba_{\bullet})$ satisfies $\phi(F^e_* a Q) \subseteq Q$.  Our goal is to show that there are finitely many prime ideals $Q$ such that $\phi(F^e_* a Q) \subseteq Q$.

First note that we can replace $\phi(\blank)$ by $\phi(a \times \blank)$ and so ignore the term $a$.
For a contradiction, assume there are infinitely many such prime ideals $Q$ such that $\phi(F^e_* Q) \subseteq Q$.   We choose a collection $\mathfrak{Q}$ of infinitely many primes ideals $Q$ satisfying:
\begin{itemize}
 \item[(i)]  $\phi(F^e_* Q) \subseteq Q$
 \item[(ii)]  All $Q \in \mathfrak{Q}$ have the same height.
\item[(iii)]  The closure of the set $\mathfrak{Q}$ in the Zariski topology is an irreducible (possibly non-proper) closed subset $W$ of $\Spec R$.  We set $P$ to be the generic point of that subset $W$ (in other words, $P = \cap_{Q \in \mathfrak{Q}} Q$).
\end{itemize}
Using the pigeon-hole principal, it is not difficult to see that a set $\mathfrak{Q}$ satisfying conditions (i), (ii) and (iii) exists.

We make two observations about the prime ideal $P$:
\begin{itemize}
 \item{}  $P$ must have smaller height than the elements of $\mathfrak{Q}$
 \item{}  $P$ satisfies $\phi(F^e_* P) \subseteq P$ since $P$ is the intersection of the elements of $\mathfrak{Q}$
\end{itemize}
By restricting to an open affine set of $\Spec R$ containing $P$, we may assume that $R/P$ is normal (the elements of $\mathfrak{Q}$ will still form a dense subset of $\Spec R/P$).  Therefore, $\phi$ induces a divisor $\Delta_P$ on $\Spec R/P$ as in Theorem \ref{ThmFirstFAdjunction}.  The set of elements in $\mathfrak{Q}$ restrict to centers of sharp $F$-purity for $(R/P, \Delta_P)$ by Theorem \ref{ThmFirstFAdjunction}(vi).  As noted above, $\{ Q/P \text{ } | \text{ } Q \in \mathfrak{Q} \}$ is dense in $\Spec R/P$ and simultaneously $\{ Q/P \text{ } | \text{ } Q \in \mathfrak{Q} \}$ is contained in the non-strongly $F$-regular locus of $(R/P, \Delta_P)$, which is closed and proper.  This is a contradiction.
\end{proof}

\begin{remark}
If one wishes to assume that $R$ is not necessarily normal and that $\Delta = 0$ (or even that $\Delta$ is some sort of appropriate generalization of a $\bQ$-divisor, see for example \cite{HartshonreGeneralizedDivisorsAndBiliaison} or \cite[Chapter 16]{KollarFlipsAndAbundance}), the proof goes through without change.
\end{remark}


\begin{corollary}
\label{CorFinitelyManyCompatiblySplit}
Suppose that $X$ is a noetherian $F$-finite Frobenius split scheme with splitting $\phi : F^e_* \O_X \rightarrow \O_X$, then there exists at most finitely many $\phi$-compatibly split subschemes.
\end{corollary}
\begin{proof}
Use a finite affine cover of $X$.  On each open affine subset, there are finitely many compatibly split subschemes by the above argument.
\end{proof}

\section{Comments on adjoint-like test ideals and restriction theorems}
\label{SectionAdjointLikeTestIdeals}

Based on the work of Takagi, it is natural to hope that there is a restriction theorem of (generalized) adjoint-like test ideals, similar to the ones in \cite{TakagiPLTAdjoint} and \cite{TakagiHigherDimensionalAdjoint}.  Using the results of the previous section, we can accomplish this.

\begin{definition}
\label{DefnPenultimateAdjointTestIdeal}
Suppose that $R$ is $F$-finite normal ring and that $(R, \Delta, \ba^t)$ is a triple.  Further suppose that :
\begin{itemize}
\item[(a)] $Q \in \Spec R$ is a center of sharp $F$-purity for $(R, \Delta)$.
\item[(b)] $\ba \cap (R \setminus Q) \neq \emptyset$.
\item[(c)] $(R_Q, \Delta|_{\Spec R_Q})$ is sharply $F$-pure.
\item[(d)] $R/Q$ is normal.
\item[(e)] There exists an integer $e_0$ such that $\Hom_R(F^{e_0}_* R((p^{e_0} - 1)\Delta), R)$ is free as an $F^{e_0}_* R$-module.
\item[(f)] The integer $e_0$ is the smallest positive integer satisfying condition (e).
\end{itemize}
Fix a map $\phi_{e_0} = \phi : F^{e_0}_* R \rightarrow R$ corresponding to $\Delta$.  We define the \emph{big test ideal of $(R, \Delta, \ba^t)$ outside of $Q$}, denoted $\tau_{b}(R; \nsubseteq Q; \Delta, \ba^t)$ (if it exists), to be the smallest ideal $J$ satisfying the following two conditions:
\begin{itemize}
 \item  $J$ is not contained in $Q$ (that is, $J \cap (R \setminus Q) \neq \emptyset$).
 \item $\phi_{ne_0}(F^{ne_0}_* \ba^{\lceil t(p^{ne_0} - 1) \rceil}  J) \subseteq J$ for all $n \geq 0$ where $\phi_{ne_0}$ is as in Definition \ref{DefinitionComposingMaps}.
\end{itemize}
\end{definition}

\begin{remark}
Note that with regards to Definition \ref{DefnPenultimateAdjointTestIdeal}(b), using the fact that $\ba \cap (R \setminus Q) \neq \emptyset$, we see that $Q$ is a center of sharp $F$-purity for $(R, \Delta)$ if and only if it is a center of sharp $F$-purity for $(R, \Delta, \ba^t)$.  Likewise, the localized pair $(R_Q, \Delta|_{\Spec R_Q})$ is sharply $F$-pure if and only if the localized triple $(R_Q, \Delta|_{\Spec R_Q}, (\ba R_Q)^t)$ is sharply $F$-pure since $\ba R_Q = R_Q$.
\end{remark}

\begin{remark}
It is unnecessary to choose $e_0$ to be the \emph{smallest} integer satisfying condition (e).  If one uses any integer $e_0$ satisfying condition (e), then one obtains the same $\tau_{b}(R, \nsubseteq Q; \Delta, \ba^t)$.  We will not verify this here as the proof is rather involved and is essentially the same argument as in Proposition \ref{PropositionBigTestIdealIsSmallest}.
\end{remark}

\begin{remark}
It is also interesting to study the smallest ideal $J$ which properly contains $Q$ and such that $\phi_{ne_0}(F^{ne_0}_*  \ba^{\lceil t(p^{ne_0} - 1)\rceil}  J) \subseteq J$ for all $n \geq 0$ (again, if it exists).  For future reference, we will denote that ideal by $\tau_{b}(R, \supseteq Q; \Delta, \ba^t)$.
\end{remark}

\begin{remark}
If $\ba = R$, then $\tau_{b}(R, \nsubseteq Q; \Delta) = \tau_{b}(R, \nsubseteq Q; \Delta, \ba^t)$ is the unique smallest ideal not contained in $Q$ such that $\phi_{e_0}(F^{e_0}_* J) \subseteq J$.  Likewise, if $\ba = R$, $\tau_{b}(R, \supseteq Q; \Delta) = \tau_{b}(R, \supseteq Q; \Delta, \ba^t)$ is the smallest ideal properly containing $Q$ such that $\phi_{e_0}(F^{e_0}_*  J) \subseteq J$.
\end{remark}

\begin{remark}
\label{RemarkOnNonIntegralCenters}
It is probably interesting to look at non-prime radical ideals $Q$ which are $F$-compatible with respect to $(R, \Delta)$.  Set $R^{\circ Q}$ to be the set of elements not contained in any minimal prime of $Q$.  In that case, one should probably consider ideals $J$  minimal with respect to the conditions that  $J \cap R^{\circ Q} \neq \emptyset$ and $\phi(F^{e_0}_* J) \subseteq J$.  If one takes $Q$ to be the zero ideal of $R$, then $\tau_{b}(R, \nsubseteq Q; \Delta)$ is just the usual big test ideal, see Proposition \ref{PropositionBigTestIdealIsSmallest}.  However, in this paper, we will not work in this generality.
\end{remark}



\begin{remark}
\label{RemarkRelationsBetweenAdjointTestIdeals}
Suppose that the ideals $\tau_{b}(R, \nsubseteq Q; \Delta, \ba^t)$ and $\tau_{b}(R, \supseteq Q; \Delta, \ba^t)$ exist.  Notice that $\tau_{b}(R, \nsubseteq Q; \Delta, \ba^t) \subseteq \tau_{b}(R, \supseteq Q; \Delta, \ba^t)$.  Furthermore, we claim that
\begin{equation}
\label{EquationRestrictionClaim}
\tau_{b}(R, \nsubseteq Q; \Delta, \ba^t) + Q = \tau_{b}(R, \supseteq Q; \Delta, \ba^t).
\end{equation}
The containment $\supseteq$ follows from the definition of $\tau_{b}(R, \supseteq Q; \Delta, \ba^t)$ because $\tau_{b}(R, \nsubseteq Q; \Delta, \ba^t) + Q$ satisfies
\begin{equation}
\label{EqnCompatiContainment}
\phi_{ne_0}(F^{ne_0}_* \ba^{\lceil t(p^{ne_0} - 1) \rceil}(\tau_{b}(R, \nsubseteq Q; \Delta, \ba^t) + Q)) \subseteq \tau_{b}(R, \nsubseteq Q; \Delta, \ba^t) + Q
 \end{equation}
since both $Q$ and $\tau_{b}(R, \nsubseteq Q; \Delta, \ba^t)$ satisfy the condition of Equation \ref{EqnCompatiContainment}.  But then since both $\tau_{b}(R, \nsubseteq Q; \Delta, \ba^t)$ and $Q$ are contained in $\tau_{b}(R, \supseteq Q; \Delta, \ba^t)$, we are done.
\end{remark}

We can now prove that $\tau_{b}(R, \supseteq Q; \Delta, \ba^t)$ exists.

\begin{proposition}
\label{PropInitialRestrictionTheoremForTestIdeals}
Suppose that $(R, \Delta, \ba^t)$ and $Q \in \Spec R$ are as in Definition \ref{DefnPenultimateAdjointTestIdeal}.  Further suppose that $\alpha : R \rightarrow R/Q$ is the natural surjection.  Suppose that $\Delta_{R/Q}$ is the $\bQ$-divisor on $\Spec R/Q$ corresponding to $\Delta$ as in Theorem \ref{ThmFirstFAdjunction}.  Then $\tau_{b}(R, \supseteq Q; \Delta, \ba^t)$ exists and is equal to $\alpha^{-1}(\tau_{b}(R/Q; \Delta_{R/Q}, \overline\ba^t))$.  In particular
\[ \tau_{b}(R, \supseteq Q; \Delta, \ba^t)/ Q = \tau_{b}(R, \supseteq Q; \Delta, \ba^t)|_{R/Q} = \tau_{b}(R/Q; \Delta_{R/Q}, \overline\ba^t). \]
\end{proposition}
\begin{proof}
As noted before, it is easy to see that if $J$ contains $Q$ and $\phi_{ne_0}(F^e_* \ba^{\lceil t(p^{ne_0} - 1) \rceil} J) \subseteq J$, then $\phi_{ne_0,Q}\left(F^e_* \overline \ba^{\lceil t(p^{ne_0} - 1) \rceil} (J/Q) \right) \subseteq J/Q$.  Conversely, if we have an ideal $J \supseteq Q$ such that $\phi_{ne_0,Q}\left(F^e_* \overline \ba^{\lceil t(p^{ne_0} - 1) \rceil} (J/Q) \right) \subseteq J/Q$ then $\phi_{ne_0}(F^e_* \ba^{\lceil t(p^{ne_0} - 1) \rceil} J) \subseteq J + Q = J$.  But ideals of $R$ containing $Q$ are in bijection with ideals of $R/Q$.  This completes the proof.
\end{proof}

Once we have verified that $\tau_{b}(R, \nsubseteq Q; \Delta, \ba^t)$ exists, Proposition \ref{PropInitialRestrictionTheoremForTestIdeals} will immediately imply the following restriction theorem.

\begin{corollary}
\label{CorRestrictionTheoremForAdjoint2}
Suppose that $(R, \Delta, \ba^t)$ and $Q \in \Spec R$ are as in Definition \ref{DefnPenultimateAdjointTestIdeal}.  Further suppose that $\Delta_{R/Q}$ is the $\bQ$-divisor on $R/Q$ corresponding to $\Delta$ as in Theorem \ref{ThmFirstFAdjunction}.
Then $\tau_{b}(R, \nsubseteq Q; \Delta, \ba^t)|_{R/Q} = (\tau_{b}(R, \nsubseteq Q; \Delta, \ba^t) + Q)|_{R/Q} = \tau_{b}(R/Q; \Delta_{R/Q}, \overline \ba^t)$.
\end{corollary}
\begin{proof}
Apply Proposition \ref{PropInitialRestrictionTheoremForTestIdeals} and Equation \ref{EquationRestrictionClaim}.  The result will follow once we know that $\tau_{b}(R, \nsubseteq Q; \Delta, \ba^t)$ exists.
\end{proof}



The rest of the section will be devoted to proving that the ideal $\tau_{b}(R, \nsubseteq Q; \Delta, \ba^t)$ exists.

\begin{remark}
\label{RemarkAlsoViaTightClosure}
One way to do this is by working out a version of tight closure theory using $c \in R \setminus Q$ instead of $c \in R^{\circ}$.  However, we will use a more direct approach.
\end{remark}

We begin with several lemmas which are essentially the same as those used in the proof the existence of test elements.  The main technical result of the section is Proposition \ref{PropositionUniformExistenceForTestElts}, which combines the following three lemmas.


\begin{lemma}
\label{LemmaExistenceOfMaps}
Suppose that $(R, \Delta)$ is a sharply $F$-pure pair, $(p^e - 1)(K_R + \Delta)$ is integral, and that $\Hom_R(F^e_* R((p^e - 1)\Delta), R)$ is free as an $F^e_* R$-module with generator $\phi_e$ (by restriction, we also view $\phi_e$ as an element of $\Hom_R(F^e_* R, R)$).  Further suppose that $d \in R$ is an element not contained in any center of $F$-purity for $(R, \Delta)$.

Then:
\begin{itemize}
\item[(i)] $1 \in \phi_{n_0e}(F^{n_0e}_* (dR) )$ for some $n_0 > 0$.
\item[(ii)]  There exists $n_0 > 0$ such that $1 \in \phi_{ne}(F^{ne}_* (dR) )$ for all $n \geq n_0$.
\end{itemize}
\end{lemma}
\begin{proof}
We begin by proving (i).
First we claim that the statement is local.  Another way to phrase the conclusion of the lemma is that $\phi_{ne}(F^{ne}_* (d R) ) = R$.  However, $\phi_{ne}(F^{ne}_* (d R) ) = R$ (for a fixed $n$) if and only if it is true after localizing at each maximal ideal.  Conversely, if $(\phi_{n_ie})_{\bm_i}(F^{n_ie}_* d R_{\bm_i} ) = R_{\bm_i}$ after localizing at some maximal ideal $\bm_{i}$ for some $n_i$, then it holds in a neighborhood of $\bm_{i}$ for the same $n_i$.  Cover $\Spec R$ by a finite number of such neighborhoods and choose a sufficiently large $n$ that works on all neighborhoods.\footnote{Note that if $1 \in \phi_{ne}(F^{ne}_* (d R) )$ then $1 \in \phi_{ne}(F^{ne}_* R )$.  By composition, this implies that $1 \in \phi_{mne}(F^{mne}_* (dR))$ for all integers $m > 0$.}  Therefore we may assume that $R = (R, \bm)$ is local.  Note that this is essentially the same as the usual proof that strong $F$-regularity localizes.

Choose a minimal center $Q$ of sharp $F$-purity for $(R, \Delta)$ and mod out by $Q$.  It follows that $(R/Q, \Delta_{R/Q})$ is strongly $F$-regular and also that $\overline{d} \neq 0 \in R/Q$.

In particular, for some $n > 0$, we have $\overline \phi_{ne}(F^{ne}_* \overline d R/Q) = R/Q$.  Therefore, we can find an element $\overline b \in R/Q$ such that $\overline \phi_{ne}(F^{ne}_* \overline d \overline b) = 1 \in R/Q$.  By choosing an arbitrary $b \in R$ such that the coset $b + Q = \overline b$, we see that $\phi_{ne}(F^{ne}_* db) = 1+x$ for some $x \in Q$.  Since $R$ is local, $Q \subseteq \bm$ and $1 + x$ is a unit, we have $1 \in \phi_{ne}(F^{ne}_* (dR))$ as desired.

We now prove (ii).  Let $n_0$ be the integer from part (i).  Note that it follows that $1 \in \phi_{n_0e}(F^{n_0e}_* R )$ so there exists an element $f \in R$ such that $1 = \phi_{n_0e}(F^{n_0e}_* f )$.  In particular, the map
\[
\xymatrix@R=6pt{
R \ar[r]^-{F^{n_0e}} & F^{n_0e}_* R((p^{n_0e} - 1)\Delta) \\
1 \ar@{|->}[r] & F^{n_0e}_* 1 \\
}
\]
splits.  This implies that $F^e : R \rightarrow F^e_* R((p^e - 1)\Delta)$ also splits.  But then $1 \in \phi_{e}(F^{e}_* R)$ since $\phi_{e}$ was chosen as a generator of $\Hom_R(F^e_* R((p^e - 1)\Delta), R)$.  Therefore we see that,
\[
1 \in \phi_{e}(F^{e}_* R) = \phi_{e}(F^{e}_* \phi_{ne}(F^{ne}_* (dR))) = \phi_{(n+1)e}(F^{(n+1)e}_* (dR) ).
\]
Repeatedly applying $\phi_e$ will then complete the proof of (ii).
\end{proof}

\begin{lemma}
\label{LemmaConstructionOfVaryingTestElement}
Suppose that $(R, \Delta, \ba^t)$ is a triple and $Q \in \Spec R$ is a center of $F$-purity satisfying the conditions from Definition \ref{DefnPenultimateAdjointTestIdeal}.  Then there exists an element $c \in R \setminus Q$ that satisfies the following condition:

For all $d \in R \setminus Q$ and for all sufficiently large $n > 0$, there exists an integer $m' > 0$ (which depends on both $n$ and $d$) such that $c^{m'} \in \phi_{ne_0}(F^{ne_0}_* d \ba^{\lceil t(p^{ne} - 1) \rceil})$.
\end{lemma}
\begin{proof}
Choose $c \in \ba \cap (R \setminus Q)$ so that
\begin{itemize}
\item[(a)] $(R_c, \Delta|_{\Spec R_c})$ is sharply $F$-pure.
\item[(b)] There are no centers of sharp $F$-purity for $(R_c, \Delta|_{\Spec R_c})$ which contain $Q R_c$ (as an ideal).
\item[(c)] All centers of sharp $F$-purity for $(R_c, \Delta|_{\Spec R_c})$ are contained in $Q R_c$ (as ideals).
\end{itemize}
In particular, $d/1 \in R_c$ is not contained in any centers of sharp $F$-purity for $(R_c, \Delta|_{\Spec R_c})$.  Note conditions (b) and (c) above may be summarized by saying that $QR_c$ is the unique maximal height (as an ideal) center of sharp $F$-purity.

Therefore, by Lemma \ref{LemmaExistenceOfMaps}, we know that for all $n \gg 0$,  $1 \in (\phi_{ne_0})_c(F^{ne_0}_* (d R_c))$.  This implies that $c^{m'} \in \phi_{ne_0} (F^{ne_0}_* d \ba^{\lceil t(p^{ne_0} - 1) \rceil})$ for some $m'$.
\end{proof}

\begin{lemma}
\label{LemmaRepeatingMapWillMakeTestElements}
Suppose that for some $e > 0$, we have a map $\gamma_e : F^e_* R \rightarrow R$ such that $b \in \gamma_e(F^e_* \ba^{\lceil t(p^e - 1) \rceil} )$.  Then for all $n > 0$, $b^2 \in \gamma_{ne}(F^{ne}_* \ba^{\lceil t(p^{ne} - 1) \rceil})$.  Here $\gamma_{ne}$ is the map obtained by composing $\gamma$ with itself $(n-1)$-times, as in Definition \ref{DefinitionComposingMaps}.
\end{lemma}
\begin{proof}
We proceed by induction. The case $n = 1$ was given by hypothesis.  Now suppose the result holds for $n$ (that is, $b^2 \in \gamma_{ne}(F^{ne}_* \ba^{\lceil t(p^{ne} - 1) \rceil})$).  However,
\[
\begin{split}
b^2 \in b \gamma_{e}(F^{e}_* \ba^{\lceil t(p^e - 1) \rceil}) = \gamma_{e}(F^{e}_* \ba^{\lceil t(p^e - 1) \rceil} b^{p^e}) \subseteq \gamma_{e}(F^{e}_* \ba^{\lceil t(p^e - 1) \rceil} b^2) \subseteq \\
  \gamma_{e}\left(F^{e}_* \ba^{\lceil t(p^e - 1) \rceil} \gamma_{ne}(F^{ne}_* \ba^{\lceil t(p^{ne} - 1) \rceil}) \right) = \gamma_{e}\left(F^{e}_* \gamma_{ne}(F^{ne}_* (\ba^{\lceil t(p^e - 1) \rceil})^{[p^{ne}]} \ba^{\lceil t(p^{ne} - 1) \rceil})\right)  \subseteq  \\
\gamma_{(n+1)e}\left(F^{(n+1)e}_* \ba^{\lceil t(p^{(n+1)e} - 1) \rceil}\right) \text{ as desired.}
\end{split}
\]
\end{proof}

We now come to the main technical result of the section.

\begin{proposition}
\label{PropositionUniformExistenceForTestElts}
Assume the notation and conventions from Definition \ref{DefnPenultimateAdjointTestIdeal}.  There is an element $b \in R \setminus Q$ such that for every $d \in R \setminus Q$, there exists an integer $n_d > 0$ such that $b \in \phi_{n_de_0}(F^{n_de_0}_* d\ba^{\lceil t(p^{n_de_0} - 1) \rceil})$.  Note that $b$ does not depend on $d$.
\end{proposition}
\begin{proof}
Fix $c \in R \setminus Q$ satisfying Lemma \ref{LemmaConstructionOfVaryingTestElement}.  Then there exist integers $n_1, m_1 > 0$ such that $c^{m_1} \in \phi_{n_1 e_0}(F^{n_1 e_0}_* (1)\ba^{\lceil t(p^{n_1 e_0} - 1) \rceil})$.  An application of Lemma \ref{LemmaRepeatingMapWillMakeTestElements} then implies that $c^{2m_1} \in \phi_{n n_1 e_0}(F^{n n_1 e_0}_* (1)\ba^{\lceil t(p^{n n_1 e_0} - 1) \rceil})$ for all $n > 0$.    We will show that $c^{3m_1} = b$ works.

Likewise, by Lemma \ref{LemmaConstructionOfVaryingTestElement}, for some $n' > 0$ there exists $m_d$ such that,
\[
c^{m_d} \in \phi_{n' e_0} (F^{n' e_0}_* (d)\ba^{\lceil t(p^{n' e_0} - 1)\rceil}).
\]
If $m_d < 3m_1$, we are done (set $n_d = n'$).  Otherwise, choose $n > 0$ such that $m_1p^{n n_1 e_0} \geq m_d$.  Then,
\[
\begin{split}
c^{3m_1} = c^{m_1} c^{2m_1} \in c^{m_1} \phi_{n n_1 e_0}\left(F^{n n_1 e_0}_* \ba^{\lceil t(p^{n n_1 e_0} - 1) \rceil}\right) = \\
\phi_{n n_1 e_0}\left(F^{n n_1 e_0}_* \ba^{\lceil t(p^{n n_1 e_0} - 1) \rceil} c^{m_1 p^{n n_1 e_0}}\right) \subseteq \phi_{n n_1 e_0}\left(F^{n n_1 e_0}_* \ba^{\lceil t(p^{n n_1 e_0} - 1) \rceil} c^{m_d}\right) \subseteq\\
\phi_{n n_1 e_0}\left(F^{n n_1 e_0}_* \ba^{\lceil t(p^{n n_1 e_0} - 1) \rceil} \phi_{n' e_0} (F^{n' e_0}_* (d)\ba^{\lceil t(p^{n' e_0} - 1)\rceil})\right) = \\
\phi_{n n_1 e_0}\left(F^{n n_1 e_0}_* \phi_{n' e_0} (F^{n' e_0}_* (d)(\ba^{\lceil t(p^{n n_1 e_0} - 1) \rceil})^{[p^{n' e_0}]} \ba^{\lceil t(p^{n' e_0} - 1)\rceil})\right) \subseteq \\
\phi_{(n n_1 +n')e_0}\left(F^{(n n_1 +n')e_0}_* (d)\ba^{\lceil t(p^{(n n_1 +n') e_0} - 1) \rceil} \right).
\end{split}
\]
Thus we can choose $n_d = n n_1 +n'$, which completes the proof.
\end{proof}

\begin{remark}
The $b$ from the previous proposition can be used as a big sharp test element for the variant of tight closure mentioned in Remark \ref{RemarkAlsoViaTightClosure}.  In fact, to prove the existence of big sharp test elements, one still has to prove Proposition \ref{PropositionUniformExistenceForTestElts} or something closely related to it.
\end{remark}

\begin{definition} \cite{HaraTakagiOnAGeneralizationOfTestIdeals}
\label{DefnAlternateAdjointTestIdeal}
Fix a $b$ as in Proposition \ref{PropositionUniformExistenceForTestElts}.  Then we define the ideal $\tld \tau(R; b, \Delta, \ba^t)$ as follows:
\[
\tld \tau(R; b, \Delta, \ba^t) := \sum_{n \geq 0} \phi_{ne_0}(F^{ne_0}_* b \ba^{\lceil t(p^{ne_0} - 1) \rceil}).
\]
Note that the sum stabilizes as a finite sum since $R$ is noetherian.
\end{definition}

We make several observations about this ideal (and then we will show it is equal to $\tau_{b}(R, \nsubseteq Q; \Delta, \ba^t)$).

\begin{lemma}
With notation as above, we have the following two results:
\begin{itemize}
\item[(i)]  $b \in \tld \tau(R; b, \Delta, \ba^t)$.  In particular, $\tld \tau(R; b, \Delta, \ba^t) \cap (R \setminus Q) \neq \emptyset$.
\item[(ii)]  For all $n' \geq 0$, $\phi_{n' e_0}\left(F^{n'e_0}_* \ba^{\lceil t(p^{n'e_0} - 1) \rceil} \tld \tau(R; b, \Delta, \ba^t)\right) \subseteq \tld \tau(R; b, \Delta, \ba^t)$.
\end{itemize}
\end{lemma}
\begin{proof}
For (i), simply set $d = b$ and apply Proposition \ref{PropositionUniformExistenceForTestElts}.  For (ii), notice we have the containment
\[
\begin{split}
\phi_{n' e_0}\left(F^{n'e_0}_* \ba^{\lceil t(p^{n'e_0} - 1) \rceil} \tld \tau(R; b, \Delta, \ba^t)\right) = \phi_{n' e_0}\left(F^{n'e_0}_* \ba^{\lceil t(p^{n'e_0} - 1) \rceil} \sum_{n \geq 0} \phi_{ne_0}(F^{ne_0}_* b \ba^{\lceil t(p^{ne_0} - 1) \rceil})\right) \subseteq \\
\phi_{n' e_0}\left(F^{n'e_0}_* \sum_{n \geq 0} \phi_{ne_0}(F^{ne_0}_* b \ba^{\lceil t(p^{(n+n')e_0} - 1) \rceil})\right) = \sum_{n \geq n'} \phi_{ne_0}(F^{ne_0}_* b \ba^{\lceil t(p^{ne_0} - 1) \rceil}) \subseteq \tld \tau(R; b, \Delta, \ba^t).
\end{split}
\]
\end{proof}

\begin{theorem}
\label{TheoremPenultimateAdjointExists}
For $b \in (R \setminus Q)$ as in Proposition \ref{PropositionUniformExistenceForTestElts}, the ideal $\tld \tau(R; b, \Delta, \ba^t)$ is the unique smallest ideal $J$ that satisfies
\begin{itemize}
\item{} $J \cap (R \setminus Q) \neq \emptyset$ and,
\item{} $\phi_{n e_0}(F^{ne_0}_* \ba^{\lceil t(p^{ne_0} - 1) \rceil} J) \subseteq J$ for all $n \geq 0$.
\end{itemize}
Therefore $\tau_{b}(R, \nsubseteq Q; \Delta, \ba^t) = \tld \tau(R; b, \Delta, \ba^t)$.
\end{theorem}
\begin{proof}
The previous lemma proves that $\tld \tau(R; b, \Delta, \ba^t)$ satisfies the two conditions.  Suppose that $J$ is any other ideal that also satisfies the two conditions in Theorem \ref{TheoremPenultimateAdjointExists}.  Choose $d \in J \cap (R \setminus Q)$.  By hypothesis,
\[
\sum_{n \geq 0}\phi_{ne_0}(F^{ne_0}_* d \ba^{\lceil t(p^{ne_0} - 1) \rceil}) \subseteq \sum_{n \geq 0}\phi_{ne_0}(F^{ne_0}_* \ba^{\lceil t(p^{ne_0} - 1) \rceil} J) \subseteq J
\]
and so by Proposition \ref{PropositionUniformExistenceForTestElts}, we see that $b \in J$.  But then
\[
\tld \tau(R; b, \Delta, \ba^t) = \sum_{n \geq 0} \phi_{ne_0}(F^{ne_0}_*b \ba^{\lceil t(p^{ne_0} - 1) \rceil}) \subseteq \sum_{n > 0}\phi_{ne_0}(F^{ne_0}_* \ba^{\lceil t(p^{ne_0} - 1) \rceil} J) \subseteq J.
\]
\end{proof}

\begin{remark}
Theorem \ref{TheoremPenultimateAdjointExists} implies that $\tld \tau(R; b, \Delta, \ba^t)$ is also independent of the choice of $b$ (as long as $b$ is chosen via Proposition \ref{PropositionUniformExistenceForTestElts}).
\end{remark}

\begin{remark}
If $b$ is as in Proposition \ref{PropositionUniformExistenceForTestElts}, then for any multiplicative set $T$, it follows that $b / 1$ satisfies Proposition \ref{PropositionUniformExistenceForTestElts} for the localized triple $(T^{-1}R, \Delta|_{\Spec T^{-1} R}, (T^{-1}\ba)^t)$.  Therefore the formation of $\tau_{b}(R, \nsubseteq Q; \Delta, \ba^t) = \tld \tau(R; b, \Delta, \ba^t)$ commutes with localization.  In particular, we can define $\tau_{b}(X, \nsubseteq W; \Delta, \ba^t)$ on a scheme $X$ with center of $F$-purity $W$ which locally satisfies the conditions of Definition \ref{DefnPenultimateAdjointTestIdeal}.
\end{remark}


\section{Comments on codimension one centers of $F$-purity}

Suppose that $(X = \Spec R, \Delta + D)$ is a pair and $D \subseteq X$ is a integral normal reduced and irreducible divisor and $\Delta$ and $D$ have no common components.  Further assume that $K_X + \Delta + D$ is $\bQ$-Cartier with index not divisible by $p > 0$.  Since $X$ is normal, $(X, \Delta + D)$ is $F$-pure at the generic point of $D$ and $D$ is also center of $F$-purity for the pair $(X, \Delta + D)$.  If we were working in characteristic zero, there is the notion of the ``different'', see \cite{KollarFlipsAndAbundance}.  If $Q$ is defining ideal of $D$, then the different is an effective divisor that plays a role similar to the divisor $\Delta_{R/Q}$ from Theorem \ref{ThmFirstFAdjunction}.

We will show that the different and $\Delta_{R/Q}$ agree under the hypothesis that $D$ is Cartier in codimension 2.  Roughly speaking, this is the case where the different is uninteresting (it is also the case discussed in \cite{KollarMori}).   We will then give two applications of the methods used to prove this result.  We expect that the different and $\Delta_{R/Q}$ coincide in general although we do not have a proof, see Remark \ref{RemFurtherCommentsOnTheDifferent}.

First we need the following lemma.  This lemma is implicit in the work we have done previously, but we provide an explicit proof for completeness.  Lemma \ref{LemmaDualizedMapIsGenerator} is also closely related to the fact that the set of Frobenius actions on $H^d_{\bm}(R)$ is generated by the natural Frobenius action $F^e : H^{\dim R}_{\bm}(R) \rightarrow H^{\dim R}_{\bm}(R)$; see \cite{LyubeznikSmithCommutationOfTestIdealWithLocalization}.

\begin{lemma}
\label{LemmaDualizedMapIsGenerator}
 Suppose that $R$ is an $F$-finite Gorenstein local ring.  By dualizing the natural map $G : R \rightarrow F^e_* R$ (apply $\Hom_R(\blank, \omega_R)$), we construct the map
\[
 \Psi : F^e_* \omega_R \rightarrow \omega_R
\]
By fixing any isomorphism of $\omega_R$ with $R$ (which we can do since $R$ is Gorenstein), we obtain a map which we also call $\Psi$,
\[
 \Psi : F^e_* R \rightarrow R.
\]
This map $\Psi$ is an $F^e_* R$-module generator of $\Hom_R(F^e_* R, R)$.  In particular, if $R$ is normal, then $\Psi$ corresponds to the divisor $0$ via \ref{TheoremMapsInduceDivisors}.
\end{lemma}
\begin{proof}
 First note that the choices we made in the setup of the lemma are all unique up to multiplication by a unit (note there is also the choice of isomorphism between $(F^e)^! \omega_R$ with $F^e_* \omega_R$ as in Remark \ref{RemHaveToBeCareful}).  Therefore, these choices are irrelevant in terms of proving that $\Psi$ is an $F^e_* R$-module generator.  Suppose that $\phi$ is an arbitrary $F^e_* R$-module generator of $\Hom_R(F^e_* R, R)$, and so we can write $\Psi(\blank) = \phi(d \cdot \blank)$ for some $d \in F^e_* R$.  Using the same isomorphisms we selected before, we can view $\phi$ as a map $F^e_* \omega_R \rightarrow \omega_R$.  By duality for a finite morphism, we obtain $\phi^{\vee} : R \rightarrow F^e_* R$.  Note now that $G(\blank) = d \cdot \phi^{\vee}(\blank)$.  But $G$ sends $1$ to $1$ which implies that $d$ is a unit and completes the proof.
\end{proof}

We now need the following (useful) surjectivity.  A similar argument (involving local duality) was used in the characteristic $p > 0$ inversion of adjunction result of \cite[Theorem 4.9]{HaraWatanabeFRegFPure}.

\begin{proposition}
\label{PropSurjectiveCodimension1Map}
Using the notation above, further suppose that $D$ is Cartier in codimension 2 and that $(p^e - 1)(K_X + D + \Delta)$ is Cartier.  Then the natural map of $F^e_* \O_X$-modules:
\[
 \Phi : \Hom_{\O_X}(F^e_* \O_X((p^e - 1)(D + \Delta)), \O_X) \rightarrow \Hom_{\O_D}(F^e_* O_D( (p^e - 1)\Delta |_D), \O_D).
\]
induced by restriction is surjective.
\end{proposition}
\begin{proof}
The statement is local so we may assume that $X = \Spec R$ where $R$ is the spectrum of a local ring. Furthermore, because we are working locally, the domain of $\Phi$ is isomorphic to $F^e_* R$.  Thus the image of $\Phi$ is cyclic as an $F^e_* \O_D$-module which implies that the image of $\Phi$ is a reflexive $F^e_* \O_D$-module.  Therefore, it is sufficient to prove that $\Phi$ is surjective at the codimension one points of $D$ (which correspond to codimension two points of $X$).  We now assume that $X = \Spec R$ is the spectrum of a two dimensional normal local ring and that $D$ is a Cartier divisor defined by a local equation $(f = 0)$.  Since $D$ is normal and one dimensional, $D$ is Gorenstein, and so $X$ is also Gorenstein.  In particular, $(p^e - 1)\Delta$ is Cartier.  This also explains how we can restrict $(p^e -1)\Delta$ to $D$:  perform the restriction at codimension 1 points of $D$, and then take the corresponding divisor.

Consider the following diagram of short exact sequences:
\[
 \xymatrix{
0 \ar[r] & R \ar[d]_{1 \mapsto f^{p^e - 1}}\ar[r]^{\times f} & R \ar[d]^{1 \mapsto 1} \ar[r] & R/f \ar[r] \ar[d]^{1 \mapsto 1} & 0\\
0 \ar[r] & F^e_* R \ar[r]^{F^e_* \times f} & F^e_* R \ar[r] & F^e_* (R/f) \ar[r] & 0.
}
\]
Apply the functor $\Hom_R(\blank, \omega_R)$ and note that we obtain the following diagram of short exact sequences.
\[
 \xymatrix{
0 \ar[r] & \omega_{R} \ar[r]^{\times f} & \omega_{R} \ar[r] & \omega_{R/f} \cong \Ext^1_R(R/f, \omega_R) \ar[r] & 0 \\
0 \ar[r] & F^e_* \omega_{R} \ar[u]^{\alpha} \ar[r]^{F^e_* \times f} & F^e_* \omega_{R} \ar[u]^{\beta} \ar[r] & F^e_* \omega_{R/f} \cong \Ext^1_R(F^e_* (R/f), \omega_R) \ar[u]^{\delta} \ar[r] & 0
}
\]
The sequences are exact on the right because $R$ is Gorenstein and hence Cohen-Macaulay.  Note that by Lemma \ref{LemmaDualizedMapIsGenerator}, we see that $\delta$ and $\alpha$ can be viewed as $F^e_*R$-module generators of the modules $\Hom_{R/f}(F^e_* (R/f), R/f) \cong \Hom_{R/f}(F^e_* \omega_{R/f}, \omega_{R/f})$ and $\Hom_R(F^e_* R, R) \cong \Hom_R(F^e_* \omega_R, \omega_R)$ respectively. Furthermore, the map labeled $\beta$ can be identified with $\alpha \circ \left(F^e_* (\times f^{p^e - 1}) \right)$.

But the diagram proves exactly that the map $\beta \in \Hom_R(F^e_* R, R)$ restricts to a generator of $\Hom_{R/f}(F^e_* \omega_{R/f}, \omega_{R/f})$ which is exactly what we wanted to prove in the case that $\Delta = 0$.  When $\Delta \neq 0$, we can simply pre-multiply the $\alpha$, $\beta$ and $\delta$ with a local generator of the Cartier divisor $(p^e - 1)\Delta$.
\end{proof}

\begin{remark}
 Suppose that $X$ is normal, $\Delta = 0$ and $D$ is Gorenstein in codimension 1 and S2 (but $D$ is not necessarily normal or irreducible), then the map $\Phi$ from Proposition \ref{PropSurjectiveCodimension1Map} is still surjective.  The proof is unchanged.
\end{remark}


The previous example also gives us the following corollary.  Compare with \cite[Theorem 5.1]{KollarShepherdBarron}, \cite[Theorem 2.5]{KaruBoundedness}, \cite[Proposition 2.13]{FedderWatanabe} and \cite[Theorem 5.1]{SchwedeEasyCharacterization}.

\begin{corollary}
Suppose that $R$ is normal, local and $\bQ$-Gorenstein with index not divisible by $p$ and that $f \in R$ is a non-zero divisor such that the map $\Phi$ from Proposition \ref{PropSurjectiveCodimension1Map} (where $D = \Div(f)$ and $\Delta = 0$) is surjective.\footnote{Note that $\Phi$ is surjective if $R/f$ is normal, or more generally if $R/f$ is S2 and Gorenstein in codimension 1.}

If $R[f^{-1}]$ is strongly $F$-regular and $R/f$ is $F$-pure then $R$ is strongly $F$-regular.  In particular, both $R$ and $R/f$ are Cohen-Macaulay.
\end{corollary}
\begin{proof}
Since the map
\[
\Phi : \Hom_R(F^e_* R((p^e - 1)\Div(f)), R) \rightarrow \Hom_{R/f}(F^e_* (R/f), R/f).
\]
is surjective, a splitting $\overline{\phi} \in \Hom_{R/f}(F^e_* (R/f), R/f)$ has a pre-image $\phi \in \Hom_R(F^e_* R, R)$.  It then follows (just as in Observation \ref{ObsMapsRestrictToCenters}) that the map $\phi$ is also surjective.  In particular, $\phi$ sends some multiple of $f^{p^e - 1}$ to $1$.  But then since $R[f^{-1}]$ is strongly $F$-regular, we see that $R$ itself is strongly $F$-regular.
\end{proof}

\begin{corollary}
 Suppose that $S$ is an $F$-finite regular local ring and $I$ is a prime ideal of $S$ such that $R = S/I$ is normal and satisfies
\[
 (I^{[p^e]} : I) = I^{[p^e]} + (g)
\]
for some $g \in S$ (note that this implies that $(p^e - 1)K_R$ is Cartier).  Further suppose that $f \in S$ is an element whose image in $R$ is non-zero and such that $R/(fR)$ is normal (or S2 and Gorenstein in codimension 1).  Then
\[
 \left((I + (f) )^{[p^e]} : (I + (f) )\right) = (I + (f) )^{[p^e]} + (f^{p^e -1} g).
\]
\end{corollary}
\begin{proof}
 If $A = S/(I + f)$, then it follows from Proposition \ref{PropSurjectiveCodimension1Map} that $\Hom_A(F^e_* A, A)$ is free of rank 1 as an $F^e_* A$-module and furthermore that a generator of $\Hom_A(F^e_* A, A)$ is obtained by multiplying a generator of $\Hom_R(F^e_* R, R)$ by $f^{p^e - 1}$ and restricting.  The result then follows from \cite[Lemma 1.6]{FedderFPureRat}.
\end{proof}

\begin{remark}
\label{RemFurtherCommentsOnTheDifferent}
Suppose that $D$ is a normal prime divisor on $X$ a normal scheme.  Further suppose that $\Delta$ is an effective $\bQ$-divisor (without common components with $D$) such that $K_X + \Delta + D$ is $\bQ$-Cartier.  Then there exists a canonically determined effective $\bQ$-divisor $\Delta_D$ on $D$ with $(K_X + \Delta + D)|_D \sim_{\bQ} K_D + \Delta_D$; see \cite[Chapter 16]{KollarFlipsAndAbundance} for a description of the construction of the different (which can be performed in any characteristic).  Furthermore, in characteristic zero, the singularities of $(X, D + \Delta)$ near $D$ are closely related to the singularities of $(D, \Delta_D)$; see for example \cite{KollarFlipsAndAbundance} and \cite{KawakitaInversion}.  We expect that the different coincides with the divisor $\Delta_{R/Q}$ we have constructed, but we do not have a proof (the problem might be quite easy if approached correctly).  One should note that we believe that the divisor called the ``different'' in \cite[Theorem 4.3]{TakagiPLTAdjoint} is $\Delta_Q$.  Again, we suspect that $\Delta_{R/Q}$ coincides with the different in general.
\end{remark}

\section{Comments on normalizing centers of $F$-purity}

In the characteristic zero setting, one obstruction to working with an arbitrary log canonical centers $W \subseteq X$ is the fact that $W$ may not be normal.  One way around this issue is to normalize the subscheme $W$ (even if $W$ is a divisor).  Therefore, it is tempting to do the same in positive characteristic.  Using Lemma \ref{LemmaLiftMapToNormalization}, one can do something like this in characteristic zero.  In particular, in Proposition \ref{PropositionPropertiesOfNormalizedRestrictedDelta} we do obtain canonically determined $\bQ$-divisors on the normalization of any center of $F$-purity.  However, a full analog of inversion of adjunction on log canonicity via normalizing centers of $F$-purity is impossible, as we will see in Example \ref{ExampleCannotNormalize}.

\begin{lemma}
\label{LemmaLiftMapToNormalization}
Suppose that $R$ is a reduced $F$-finite ring and that $\phi \in \Hom_R(F^e_* R, R)$.  Set $R^N$ to be the normalization of $R$ inside the total field of fractions.  Then $\phi$ extends to a unique $R^N$-linear map $\phi^N : F^e_* R^N \rightarrow R^N$ that restricts back to $\phi$.
\end{lemma}
\begin{proof}
To construct $\phi^N$, simply tensor $\phi$ with the total field of fractions $k(R)$ of $R$ and then restrict the domain to $F^e_* R^N$.  The fact that the image of $\phi^N$ is contained inside $R^N$ follows from \cite[Hint to Exercise 1.2.E(4)]{BrionKumarFrobeniusSplitting}; for a complete proof see \cite[Proposition 7.11]{SchwedeCentersOfFPurity}.  The fact that this $\phi^N$ is unique follows from the fact that the natural map
\[
\Hom_R(F^e_* R, R) \rightarrow \Hom_R(F^e_* R, R) \tensor_R k(R) \cong \Hom_{k(R)}(F^e_* k(R), k(R))
\]
is injective.
\end{proof}

\begin{proposition}
\label{PropositionPropertiesOfNormalizedRestrictedDelta}
Suppose that $X = \Spec R$ and $(X, \Delta)$ is a pair and that $\sHom_{\O_X}(F^e_* \O_X((p^e - 1)\Delta), \O_X)$ is free of rank 1 as an $F^e_* \O_X$-module.  Further suppose that $\Spec R/I = W \subset X$ is a reduced closed subscheme such that $(X, \Delta)$ is sharply $F$-pure at the generic points of $W$ and $I$ is  $F$-compatible with respect to $(R, \Delta)$.  Set $\eta : \left( \Spec R/I \right)^N = W^N \rightarrow W$ to be the normalization map and write $W^N = \coprod_{i = 1}^m W_i^N$; the disjoint union of $W^N$ into its irreducible components.

Then there exists a canonically determined $\bQ$-divisor $\Delta_{W^N}$ on $W^N$ satisfying the following properties:
\begin{itemize}
\item[(i)]  If one sets $\Delta_{W^N, i}$ to be the portion of $\Delta_{W^N}$ on $W^N_i$, then $(p^e - 1)(K_{W^N_i} + \Delta_{W^N, i})$ is Cartier and furthermore $\sHom_{\O_{W_i^N}}(F^e_* \O_{W_i^N}((p^e - 1)\Delta_{W^N, i}), \O_{W_i^N}) \cong F^e_* \O_{W_i^N}$ as $F^e_* \O_{W_i^N}$-modules.
\item[(ii)]  The conductor ideal of $(R/I)$ in $(R/I)^N$ is $F$-compatible with respect to $((R/I)^N, \Delta_{W^N})$.
\item[(iii)]  The big test ideal $\tau_b( (R/I)^N; \Delta_{W^N})$ of $((R/I)^N, \Delta_{W^N})$ is contained in the conductor ideal of $R/I \subseteq (R/I)^N$.
\item[(iv)] If $(X, \Delta)$ is sharply $F$-pure, then $(W^N, \Delta_{W^N})$ is also sharply $F$-pure.
\item[(v)]  If $\overline J$ is an ideal of $(R/I)^N$ which is $F$-compatible with respect to $(R, \Delta_{W^N})$, then inverse image $J$ of $\overline J$ in $R$ is $F$-compatible with respect to $(R, \Delta)$.  (In particular, $\tau_{b}(R, \nsubseteq I; \Delta)$, defined as suggested in Remark \ref{RemarkOnNonIntegralCenters}, is contained in the inverse image of $\tau_{b}( (R/I)^N, \Delta_{W^N})$).
\end{itemize}
\end{proposition}

\begin{remark}
Even though $W^N$ is not necessarily equidimensional, it is easy to define $K_{W^N}$ since we can work on each component individually.
\end{remark}

\begin{proof}
We can associate to $\Delta$ a map $\phi : F^e_* \Hom_R(F^e_* R, R)$ (up to scaling by a unit).  By assumption, this $\phi$ restricts to a map $\phi_I \in \Hom_{R/I}(F^e_* (R/I), R/I)$ which is non-zero at the generic point of each irreducible component of $R/I$.  By Lemma \ref{LemmaLiftMapToNormalization}, this map extends to a map $\phi_I^N \in \Hom_{\O_{W^N}}(F^e_* \O_{W^N}, \O_{W^N})$.  Thus this map gives us our $\Delta_{W^N}$ by Theorem \ref{TheoremMapsInduceDivisors}.  Notice that the image of a unit under $R \rightarrow (R/I)^N$ is still a unit, so that $\Delta_{W^N}$ is uniquely determined.

At this point, statement (i) is obvious.  Statement (ii) follows from \cite[Proposition 7.10]{SchwedeCentersOfFPurity} and statement (iii) follows from the fact that the big test ideal is the smallest ideal $F$-compatible ideal with respect to $((R/I)^N, \Delta_{W^N})$.  For statement (iv), note that if $\phi$ is surjective, then so is $\phi_I$.  But then it is easy to see that $\phi_I^N$ is also surjective.


To prove (v), we first note that $\phi_I(F^e_* (\overline{J} \cap R/I) ) \subseteq \overline{J} \cap R/I$.  But then we see that the pre-image of $\overline{J} \cap R/I$ in $R$ is $F$-compatible with respect to $(R, \Delta)$.
\end{proof}


One might hope that the converse to property (iv) of Proposition \ref{PropositionPropertiesOfNormalizedRestrictedDelta} above holds, but unfortunately, this is not the case.  Of course, it is easy to see that if $\phi_I^N$ is actually a splitting (ie, if it sends $1$ to $1$), then so is $\phi_I$ and thus $\phi$ is surjective near $I$ (which would imply that $(R, \Delta)$ is sharply $F$-pure near $I$).  However, it can happen that $\phi_I^N$ is surjective (that is, it sends some $x$ to $1$) but $\phi_I$ is not (in particular, the element $x$ is in $(R/I)^N$ but \emph{not} in $R/I$).  The following example illustrates this phenomenon.

\begin{example}
\label{ExampleCannotNormalize}
Suppose that $R = k[a,b,c]$ where $k = \bF_2$, the field with two elements (any perfect field of characteristic two will work).  Set $I = (ac^2 + b^2)$.  Set $\Delta = \divisor (I)$.  It is easy to see that $I$ is $F$-compatible with respect to $(R, \Delta)$.  Notice that we can write
\[
R/I = k[a,b,c]/(ac^2 + b^2) \cong k[x^2, xy, y].
\]
Therefore, the normalization of $R/I$ is simply $k[x,y]$.  We will exhibit a map $\phi_I : F_* (R/I) \rightarrow R/I$, restricted from a map $\phi : F_* R \rightarrow R$, that is not surjective, but that the extension $\phi_I^N$ to the normalization is surjective.  Of course, $R/I$ is not weakly normal and so it is not $F$-pure, which implies that no such $\phi_*$ can be surjective.

To construct $\phi$, we simply take the following map which is associated to $\Delta$.  Explicitly, we take the map
$\psi : F_* R \rightarrow R$ that sends $abc$ to $1$ (and all other lower-degree monomials to zero) and pre-compose with multiplication by $ac^2 + b^2$.  That is,
\[
\phi(\blank) = \psi\left((ac^2 + b^2) \cdot \blank \right).
\]
We compute $\phi$ on the relevant monomials.
\begin{center}
\begin{tabular}{ccccc}
$\phi(1) = 0$ & &  $\phi(c) = 0$ & & $\phi(bc) = c$\\
$\phi(a) = 0$ & &  $\phi(ab) = 0$ & & $\phi(abc) = b$\\
$\phi(b) = 0$ & &  $\phi(ac) = 0$ & & \\
\end{tabular}
\end{center}
Thus we see that $\phi$ (and therefore also $\phi_I$) is not surjective when localized at the origin.  Now we wish to consider the corresponding map on $k[x,y]$.  First we retranslate $\phi$ in terms of the variables $x$ and $y$.
\begin{center}
\begin{tabular}{ccccc}
$\phi_I^N(1) = 0$ & & $\phi_I^N(y) = 0$ & & $\phi_I^N(xy^2) = y$ \\
$\phi_I^N(x^2) = 0$ & & $\phi_I^N(x^3y) = 0$ & & $\phi_I^N(x^3y^2) = x y$ \\
$\phi_I^N(xy) = 0$ & & $\phi_I^N(x^2) = 0$ & & \\
\end{tabular}
\end{center}
Therefore, $y = \phi_I^N(xy^2) = y\phi_I^N(x)$ which implies that $\phi_I^N(x) = 1$.
\end{example}

\begin{remark}
Of course, in the above example, there were certain purely-inseparable field extensions in the normalization.  In particular, $R/I$ was
not weakly normal.  It may be that without such pure-inseparability, when $\phi_I^N$ is surjective so is $\phi$.
\end{remark}

\section{Further remarks and questions}
\label{SectionFurtherRemarks}

We conclude with some remarks and speculation.

\begin{remark}
\label{RemarkFinalNonNormal}
It is natural to try to generalize the results of this paper outside of the case when $R$ is normal.  One approach to this is to normalize $R$ as we discussed in the previous section.  However, as we saw, this approach has limitations.  Another more direct approach might be, instead of working with pairs $(R, \Delta)$ such that $(p^e -1)(K_R + \Delta)$ is Cartier, to consider pairs $(R, N)$ where $N$ is a free (or perhaps locally free) subsheaf of $\Hom_R(F^e_* R, R)$ for some $e$.

Perhaps yet a better formulation would be to consider first the graded non-commutative algebra $\oplus_e \Hom_R(F^e_* R, R)$ where the multiplication is defined by composition.  That is, for $\phi \in \Hom_R(F^d_* R, R)$ and $\psi \in \Hom_R(F^e_* R, R)$ the product $\phi \cdot \psi$ is defined to be $\phi \circ F^d_* \psi \in \Hom_R(F^{e+d}_* R, R)$.   Dually, one could consider the non-commutative ring $\mathcal{F}(E_R)$ of \cite{LyubeznikSmithCommutationOfTestIdealWithLocalization}.  Then perhaps a pair could be the combined data of the ring $R$ and a graded subalgebra $A \subseteq \oplus_e \Hom_R(F^e_* R, R)$ such that $A$ is generated as an algebra over $A_0 \cong R$ by a single element $\phi \in \Hom_R(F^e_* R, R)$ for some $e$.  Two pairs $(R, A)$ and $(R, A')$ would be said to be equivalent, if there is an integer $n > 0$ such that $A_{ne} = A'_{ne}$ for all $e$ (here $A_{ne}$ is the $ne$'th graded piece of $A$).

Almost all of the results of this paper can be generalized to such a setting.
\end{remark}

\begin{remark}
This theory can also be used to help identify subschemes of a quasi-projective variety $X$ that are compatibly split with a given Frobenius splitting.  In particular, suppose that $\phi : F^e_* \O_X \rightarrow \O_X$ is a Frobenius splitting.  We can then associate a divisor $\Delta_{\phi}$ to $\phi$.  Any center of log canonicity of the pair $(X, \Delta)$ is a center of sharp $F$-purity, see \cite{SchwedeCentersOfFPurity}, and thus the associated scheme is compatibly split with $\phi$.
\end{remark}

\begin{question}
Suppose that $R$ is a normal $\bQ$-Gorenstein ring of finite type over a field of characteristic zero and that $Q \in \Spec R$ is a center of log canonicity.  Further suppose that $R_Q$ is log canonical and that, when reduced to characteristic $p \gg 0$ (or perhaps to infinitely many $p > 0$), $({R_p})_{Q_p}$ is $F$-pure.  Then for each $p \gg 0$, we can associate a (canonically defined) $\Delta_{Q_p}$ on $R_p / Q_p$.  We then ask whether or not $\Delta_{Q_p}$ is reduced from some $\bQ$-divisor $\Delta$ on $R$?
\end{question}

\begin{question}
Is there a characteristic zero analog of $\tau_{b}(R, \nsubseteq Q; \Delta)$?  Takagi has considered similar questions, see \cite[Conjecture 2.8]{TakagiHigherDimensionalAdjoint}.  One possible analog is something along the following lines:
For a log resolution $\pi : \tld X \rightarrow X = \Spec R$ of $(R, \Delta)$, let $E = \sum E_i$ be the sum of divisors $E_i$ of $\tld X$ (exceptional or not) such that $Q \in \pi(E_i)$ and such that the discrepancy of $(R, \Delta)$ along $E_i$ is $\leq -1$.  Then consider the ideal
\[
\pi_* \O_{\tld X}(\lceil K_{\tld X} - \pi^* (K_X + \Delta) + \epsilon \sum E_i\rceil) \text{ for $\epsilon > 0$ sufficiently small.}
\]
Is it possible that this coincides with $\tau_{b}(R, \nsubseteq Q; \Delta)$ for infinitely many $p > 0$?  Also compare with \cite{FujinoNonLCSheaves}.
\end{question}

Finally, we consider the non-local setting.

\begin{remark}
\label{RemarkGlobalGluing}
Suppose that $(X, \Delta)$ is a pair where $X$ is a (possibly proper) variety of finite type over an $F$-finite field $k$.  In particular, we know that $(F^e)^! \omega_X$ can itself be identified with $\omega_X$; see Remark \ref{RemHaveToBeCareful}.  Further suppose that $K_X + \Delta$ is $\bQ$-Cartier with index not divisible by $p > 0$.  Now suppose that $W \subset X$ is a normal closed variety defined by an ideal sheaf $I_W$ which is locally $F$-compatible with respect to $\Delta$.  Then on a sufficiently fine affine cover $U_i$ of $X$, we can associate $\bQ$-divisors $\Delta_{W_i}$ on $W_i = U_i \cap W$.  It is easy to see that these divisors agree on overlaps since they were canonically determined.  Therefore, there is a $\bQ$-divisor $\Delta_W$ on $W$ determined by $(X, \Delta)$.

Furthermore, we claim that
\begin{equation}
(p^e - 1)( K_X + \Delta)|_W \sim  (p^e - 1)(K_W + \Delta_W).
\end{equation}


One way to see this is to work globally (in particular, partially globalize Theorems \ref{TheoremMapsInduceDivisors} and \ref{TheoremDivisorsInduceMaps}).  More precisely, there is a bijection of sets:
\begin{equation}
\label{EqnGlobalBijection}
\left\{ \begin{matrix}\text{Effective $\bQ$-divisors $\Delta$ on $X$ such}\\\text{that $(p^e - 1)(K_X + \Delta)$ is Cartier}\end{matrix} \right\} \leftrightarrow \left\{ \begin{matrix}\text{Line bundles $\sL$ and non-zero }\\ \text{elements of $\Hom_{\O_X}(F^e_* \sL, \O_X)$} \end{matrix}\right\} \Big/ \sim
\end{equation}
The equivalence relation on the right side identifies two maps $\phi_1: F^e_* \sL_1 \rightarrow \O_X$ and $\phi_2 : F^e_* \sL_2 \rightarrow \O_X$ if there is an isomorphism $\gamma : \sL_1 \rightarrow \sL_2$ and a commutative diagram:
\begin{equation*}
\begin{split}
\xymatrix{
 F^e_* \sL_1 \ar[d]_{F^e_* \gamma} \ar[r]^{\phi_1} & \O_X \ar[d]^{\text{id}} \\
 F^e_* \sL_2 \ar[r]^{\phi_2} & \O_X \\
}
\end{split}
\end{equation*}
We sketch the correspondence for the convenience of the reader.  Given $\Delta$, set $\sL = \O_X((1-p^e)(K_X + \Delta))$.  Then observe that
\[
\sHom_{\O_X}(F^e_*\sL,\O_X) \cong F_*^e\sHom_{\O_X}(\sL,\O_X((1-p^e)K_X)) \cong F^e_* \O_X((p^e - 1)\Delta).
\]
We can choose a global section of $\O_X((p^e - 1)\Delta)$ corresponding to the effective integral divisor $(p^e - 1)\Delta$ (up to multiplication by a unit).  This section may be viewed as a map $\phi_{\Delta} : F^e_* \sL \rightarrow \O_X$ by the above isomorphism.  For the converse direction, given such a $\phi$ we obtain a global section of $F^e_* \sL^{-1}((1-p^e)K_X)$.  This corresponds to an effective divisor $D$.  Set $\Delta_{\phi} = {1 \over p^e - 1} D$.  Again, as mentioned before, this is simply the globalized version of Theorems \ref{TheoremMapsInduceDivisors} and \ref{TheoremDivisorsInduceMaps}.

Now, since $I_W$ is locally $F$-compatible with respect to $\Delta$, we have that $\phi_{\Delta}(F^e_* I_W \sL) \subseteq I_W$.  By restriction, we obtain a map $\phi_W : \sL|_W \rightarrow \O_W$.  It is then clear that $\O_X( (p^e - 1)( K_X + \Delta) )|_W$ is isomorphic to $\O_W( (p^e - 1)(K_W + \Delta_W))$ as desired.
\end{remark}


\begin{thebibliography}{KMM87}

\bibitem[Amb99]{AmbroAdjunctionConjecture}
{\sc F.~Ambro}: \emph{The adjunction conjecture and its applications}.
  math.AG/9903060.

\bibitem[BMS08]{BlickleMustataSmithDiscretenessAndRationalityOfFThresholds}
{\sc M.~Blickle, M.~Musta{\c{t}}{\u{a}}, and K.~Smith}: \emph{Discreteness and
  rationality of {F}-thresholds}, Michigan Math. J. \textbf{57} (2008), 43--61.

\bibitem[Bou98]{BourbakiCommutativeAlgebraTranslation}
{\sc N.~Bourbaki}: \emph{Commutative algebra. {C}hapters 1--7}, Elements of
  Mathematics (Berlin), Springer-Verlag, Berlin, 1998, Translated from the
  French, Reprint of the 1989 English translation. {\sf\scriptsize MR1727221
  (2001g:13001)}

\bibitem[BK05]{BrionKumarFrobeniusSplitting}
{\sc M.~Brion and S.~Kumar}: \emph{Frobenius splitting methods in geometry and
  representation theory}, Progress in Mathematics, vol. 231, Birkh\"auser
  Boston Inc., Boston, MA, 2005. {\sf\scriptsize MR2107324 (2005k:14104)}

\bibitem[EM04]{EinMustata}
{\sc L.~Ein and M.~Musta{\c{t}}{\v{a}}}: \emph{Inversion of adjunction for
  local complete intersection varieties}, Amer. J. Math. \textbf{126} (2004),
  no.~6, 1355--1365. {\sf\scriptsize MR2102399 (2005j:14020)}

\bibitem[EM09]{EinMustataJetSchemesAndSings}
{\sc L.~Ein and M.~Musta{\c{t}}{\u{a}}}: \emph{Jet schemes and singularities},
  Algebraic geometry---{S}eattle 2005. {P}art 2, Proc. Sympos. Pure Math.,
  vol.~80, Amer. Math. Soc., Providence, RI, 2009, pp.~505--546.
  {\sf\scriptsize MR2483946}

\bibitem[EMY03]{EinMustataYasuda}
{\sc L.~Ein, M.~Musta{\c{t}}{\u{a}}, and T.~Yasuda}: \emph{Jet schemes, log
  discrepancies and inversion of adjunction}, Invent. Math. \textbf{153}
  (2003), no.~3, 519--535. {\sf\scriptsize MR2000468 (2004f:14028)}

\bibitem[EH08]{EnescuHochsterTheFrobeniusStructureOfLocalCohomology}
{\sc F.~Enescu and M.~Hochster}: \emph{The {F}robenius structure of local
  cohomology}, Algebra Number Theory \textbf{2} (2008), no.~7, 721--754.
  {\sf\scriptsize MR2460693 (2009i:13009)}

\bibitem[Fed83]{FedderFPureRat}
{\sc R.~Fedder}: \emph{{$F$}-purity and rational singularity}, Trans. Amer.
  Math. Soc. \textbf{278} (1983), no.~2, 461--480. {\sf\scriptsize MR701505
  (84h:13031)}

\bibitem[FW89]{FedderWatanabe}
{\sc R.~Fedder and K.~Watanabe}: \emph{A characterization of {$F$}-regularity
  in terms of {$F$}-purity}, Commutative algebra (Berkeley, CA, 1987), Math.
  Sci. Res. Inst. Publ., vol.~15, Springer, New York, 1989, pp.~227--245.
  {\sf\scriptsize MR1015520 (91k:13009)}

\bibitem[Fuj08]{FujinoNonLCSheaves}
{\sc O.~Fujino}: \emph{Theory of non-lc ideal sheaves--basic properties--},
  arXiv:0801.2198.

\bibitem[Har01]{HaraInterpretation}
{\sc N.~Hara}: \emph{Geometric interpretation of tight closure and test
  ideals}, Trans. Amer. Math. Soc. \textbf{353} (2001), no.~5, 1885--1906
  (electronic). {\sf\scriptsize MR1813597 (2001m:13009)}

\bibitem[Har05]{HaraACharacteristicPAnalogOfMultiplierIdealsAndApplications}
{\sc N.~Hara}: \emph{A characteristic {$p$} analog of multiplier ideals and
  applications}, Comm. Algebra \textbf{33} (2005), no.~10, 3375--3388.
  {\sf\scriptsize MR2175438 (2006f:13006)}

\bibitem[HT04]{HaraTakagiOnAGeneralizationOfTestIdeals}
{\sc N.~Hara and S.~Takagi}: \emph{On a generalization of test ideals}, Nagoya
  Math. J. \textbf{175} (2004), 59--74. {\sf\scriptsize MR2085311
  (2005g:13009)}

\bibitem[HW02]{HaraWatanabeFRegFPure}
{\sc N.~Hara and K.-I. Watanabe}: \emph{F-regular and {F}-pure rings vs. log
  terminal and log canonical singularities}, J. Algebraic Geom. \textbf{11}
  (2002), no.~2, 363--392. {\sf\scriptsize MR1874118 (2002k:13009)}

\bibitem[HY03]{HaraYoshidaGeneralizationOfTightClosure}
{\sc N.~Hara and K.-I. Yoshida}: \emph{A generalization of tight closure and
  multiplier ideals}, Trans. Amer. Math. Soc. \textbf{355} (2003), no.~8,
  3143--3174 (electronic). {\sf\scriptsize MR1974679 (2004i:13003)}

\bibitem[Har66]{HartshorneResidues}
{\sc R.~Hartshorne}: \emph{Residues and duality}, Lecture notes of a seminar on
  the work of A. Grothendieck, given at Harvard 1963/64. With an appendix by P.
  Deligne. Lecture Notes in Mathematics, No. 20, Springer-Verlag, Berlin, 1966.
  {\sf\scriptsize MR0222093 (36 \#5145)}

\bibitem[Har77]{Hartshorne}
{\sc R.~Hartshorne}: \emph{Algebraic geometry}, Springer-Verlag, New York,
  1977, Graduate Texts in Mathematics, No. 52. {\sf\scriptsize MR0463157 (57
  \#3116)}

\bibitem[Har94]{HartshorneGeneralizedDivisorsOnGorensteinSchemes}
{\sc R.~Hartshorne}: \emph{Generalized divisors on {G}orenstein schemes},
  Proceedings of Conference on Algebraic Geometry and Ring Theory in honor of
  Michael Artin, Part III (Antwerp, 1992), vol.~8, 1994, pp.~287--339.
  {\sf\scriptsize MR1291023 (95k:14008)}

\bibitem[Har07]{HartshonreGeneralizedDivisorsAndBiliaison}
{\sc R.~Hartshorne}: \emph{Generalized divisors and biliaison}, Illinois J.
  Math. \textbf{51} (2007), no.~1, 83--98 (electronic). {\sf\scriptsize
  MR2346188}

\bibitem[Hoc07]{HochsterFoundations}
{\sc M.~Hochster}: \emph{Foundations of tight closure theory}, lecture notes
  from a course taught on the University of Michigan Fall 2007.  Available online at
  {\tt{http://www.math.lsa.umich.edu/\~{ }hochster/711F07/711.html}}.

\bibitem[HH89]{HochsterHunekeTightClosureAndStrongFRegularity}
{\sc M.~Hochster and C.~Huneke}: \emph{Tight closure and strong
  {$F$}-regularity}, M\'em. Soc. Math. France (N.S.) (1989), no.~38, 119--133,
  Colloque en l'honneur de Pierre Samuel (Orsay, 1987). {\sf\scriptsize
  MR1044348 (91i:13025)}

\bibitem[HH90]{HochsterHunekeTC1}
{\sc M.~Hochster and C.~Huneke}: \emph{Tight closure, invariant theory, and the
  {B}rian\c con-{S}koda theorem}, J. Amer. Math. Soc. \textbf{3} (1990), no.~1,
  31--116. {\sf\scriptsize MR1017784 (91g:13010)}

\bibitem[HR76]{HochsterRobertsFrobeniusLocalCohomology}
{\sc M.~Hochster and J.~L. Roberts}: \emph{The purity of the {F}robenius and
  local cohomology}, Advances in Math. \textbf{21} (1976), no.~2, 117--172.
  {\sf\scriptsize MR0417172 (54 \#5230)}

\bibitem[Kar00]{KaruBoundedness}
{\sc K.~Karu}: \emph{Minimal models and boundedness of stable varieties}, J.
  Algebraic Geom. \textbf{9} (2000), no.~1, 93--109. {\sf\scriptsize MR1713521
  (2001g:14059)}

\bibitem[Kaw07]{KawakitaInversion}
{\sc M.~Kawakita}: \emph{Inversion of adjunction on log canonicity}, Invent.
  Math. \textbf{167} (2007), no.~1, 129--133. {\sf\scriptsize MR2264806
  (2008a:14025)}

\bibitem[Kaw08]{KawakitaComparisonNonLCI}
{\sc M.~Kawakita}: \emph{On a comparison of minimal log discrepancies in terms
  of motivic integration}, J. Reine Angew. Math. \textbf{620} (2008), 55--65.
  {\sf\scriptsize MR2427975}

\bibitem[Kaw97a]{KawamataOnFujitasFreenessConjectureFor3Folds}
{\sc Y.~Kawamata}: \emph{On {F}ujita's freeness conjecture for {$3$}-folds and
  {$4$}-folds}, Math. Ann. \textbf{308} (1997), no.~3, 491--505.
  {\sf\scriptsize MR1457742 (99c:14008)}

\bibitem[Kaw97b]{KawamataSubadjunctionOne}
{\sc Y.~Kawamata}: \emph{Subadjunction of log canonical divisors for a
  subvariety of codimension {$2$}}, Birational algebraic geometry ({B}altimore,
  {MD}, 1996), Contemp. Math., vol. 207, Amer. Math. Soc., Providence, RI,
  1997, pp.~79--88. {\sf\scriptsize MR1462926 (99a:14024)}

\bibitem[Kaw98]{KawamataSubadjunction2}
{\sc Y.~Kawamata}: \emph{Subadjunction of log canonical divisors. {II}}, Amer.
  J. Math. \textbf{120} (1998), no.~5, 893--899. {\sf\scriptsize MR1646046
  (2000d:14020)}

\bibitem[KMM87]{KawamataMatsudeMatsuki}
{\sc Y.~Kawamata, K.~Matsuda, and K.~Matsuki}: \emph{Introduction to the
  minimal model problem}, Algebraic geometry, Sendai, 1985, Adv. Stud. Pure
  Math., vol.~10, North-Holland, Amsterdam, 1987, pp.~283--360. {\sf\scriptsize
  MR946243 (89e:14015)}

\bibitem[KSB88]{KollarShepherdBarron}
{\sc J.~Koll{\'a}r and N.~I. Shepherd-Barron}: \emph{Threefolds and
  deformations of surface singularities}, Invent. Math. \textbf{91} (1988),
  no.~2, 299--338. {\sf\scriptsize MR922803 (88m:14022)}

\bibitem[K+92]{KollarFlipsAndAbundance}
{\sc J.~Koll{\'a}r and {14 coauthors}}: \emph{Flips and abundance for algebraic
  threefolds}, Soci\'et\'e Math\'ematique de France, Paris, 1992, Papers from
  the Second Summer Seminar on Algebraic Geometry held at the University of
  Utah, Salt Lake City, Utah, August 1991, Ast\'erisque No. 211 (1992).
  {\sf\scriptsize MR1225842 (94f:14013)}

\bibitem[KM98]{KollarMori}
{\sc J.~Koll{\'a}r and S.~Mori}: \emph{Birational geometry of algebraic
  varieties}, Cambridge Tracts in Mathematics, vol. 134, Cambridge University
  Press, Cambridge, 1998, With the collaboration of C. H. Clemens and A. Corti,
  Translated from the 1998 Japanese original. {\sf\scriptsize MR1658959
  (2000b:14018)}

\bibitem[Kun86]{KunzKahlerDifferentials}
{\sc E.~Kunz}: \emph{K\"ahler differentials}, Advanced Lectures in Mathematics,
  Friedr. Vieweg \& Sohn, Braunschweig, 1986. {\sf\scriptsize MR864975
  (88e:14025)}

\bibitem[LS01]{LyubeznikSmithCommutationOfTestIdealWithLocalization}
{\sc G.~Lyubeznik and K.~E. Smith}: \emph{On the commutation of the test ideal
  with localization and completion}, Trans. Amer. Math. Soc. \textbf{353}
  (2001), no.~8, 3149--3180 (electronic). {\sf\scriptsize MR1828602
  (2002f:13010)}

\bibitem[MR85]{MehtaRamanathanFrobeniusSplittingAndCohomologyVanishing}
{\sc V.~B. Mehta and A.~Ramanathan}: \emph{Frobenius splitting and cohomology
  vanishing for {S}chubert varieties}, Ann. of Math. (2) \textbf{122} (1985),
  no.~1, 27--40. {\sf\scriptsize MR799251 (86k:14038)}

\bibitem[MS91]{MehtaSrinivasFPureSurface}
{\sc V.~B. Mehta and V.~Srinivas}: \emph{Normal {$F$}-pure surface
  singularities}, J. Algebra \textbf{143} (1991), no.~1, 130--143.
  {\sf\scriptsize MR1128650 (92j:14044)}

\bibitem[MK09]{KumarMehtaFiniteness}
{\sc V.~B. Mehta and S.~Kumar}: \emph{Finiteness of the number of
  compatibly-split subvarieties}, arXiv:0901.2098, to appear in Int. Math. Res.
  Not.

\bibitem[Sch07]{SchwedeEasyCharacterization}
{\sc K.~Schwede}: \emph{A simple characterization of {D}u {B}ois
  singularities}, Compos. Math. \textbf{143} (2007), no.~4, 813--828.
  {\sf\scriptsize MR2339829}

\bibitem[Sch08a]{SchwedeCentersOfFPurity}
{\sc K.~Schwede}: \emph{Centers of {$F$}-purity}, arXiv:0807.1654, to appear in
  Mathematische Zeitschrift.

\bibitem[Sch08b]{SchwedeSharpTestElements}
{\sc K.~Schwede}: \emph{Generalized test ideals, sharp {$F$}-purity, and sharp
  test elements}, Math. Res. Lett. \textbf{15} (2008), no.~6, 1251--1261.
  {\sf\scriptsize MR2470398}

\bibitem[Sha07]{SharpGradedAnnihilatorsOfModulesOverTheFrobeniusSkewPolynomial%
Ring}
{\sc R.~Y. Sharp}: \emph{Graded annihilators of modules over the {F}robenius
  skew polynomial ring, and tight closure}, Trans. Amer. Math. Soc.
  \textbf{359} (2007), no.~9, 4237--4258 (electronic). {\sf\scriptsize
  MR2309183 (2008b:13006)}

\bibitem[Smi00]{SmithMultiplierTestIdeals}
{\sc K.~E. Smith}: \emph{The multiplier ideal is a universal test ideal}, Comm.
  Algebra \textbf{28} (2000), no.~12, 5915--5929, Special issue in honor of
  Robin Hartshorne. {\sf\scriptsize MR1808611 (2002d:13008)}

\bibitem[Tak04a]{TakagiInversion}
{\sc S.~Takagi}: \emph{F-singularities of pairs and inversion of adjunction of
  arbitrary codimension}, Invent. Math. \textbf{157} (2004), no.~1, 123--146.
  {\sf\scriptsize MR2135186}

\bibitem[Tak04b]{TakagiInterpretationOfMultiplierIdeals}
{\sc S.~Takagi}: \emph{An interpretation of multiplier ideals via tight
  closure}, J. Algebraic Geom. \textbf{13} (2004), no.~2, 393--415.
  {\sf\scriptsize MR2047704 (2005c:13002)}

\bibitem[Tak07]{TakagiHigherDimensionalAdjoint}
{\sc S.~Takagi}: \emph{Adjoint ideals along closed subvarieties of higher
  codimension}, arXiv:0711.2342 (2007), To appear in Crelle.

\bibitem[Tak08]{TakagiPLTAdjoint}
{\sc S.~Takagi}: \emph{A characteristic {$p$} analogue of plt singularities and
  adjoint ideals}, Math. Z. \textbf{259} (2008), no.~2, 321--341.
  {\sf\scriptsize MR2390084 (2009b:13004)}

\bibitem[TW04]{TakagiWatanabeFPureThresh}
{\sc S.~Takagi and K.-i. Watanabe}: \emph{On {F}-pure thresholds}, J. Algebra
  \textbf{282} (2004), no.~1, 278--297. {\sf\scriptsize MR2097584
  (2006a:13010)}

\end{thebibliography}

\def\cprime{$'$} \def\cprime{$'$}
  \def\cfudot#1{\ifmmode\setbox7\hbox{$\accent"5E#1$}\else
  \setbox7\hbox{\accent"5E#1}\penalty 10000\relax\fi\raise 1\ht7
  \hbox{\raise.1ex\hbox to 1\wd7{\hss.\hss}}\penalty 10000 \hskip-1\wd7\penalty
  10000\box7}
\providecommand{\bysame}{\leavevmode\hbox to3em{\hrulefill}\thinspace}
\providecommand{\MR}{\relax\ifhmode\unskip\space\fi MR}
\providecommand{\MRhref}[2]{%
  \href{http://www.ams.org/mathscinet-getitem?mr=#1}{#2}
}
\providecommand{\href}[2]{#2}

\end{document}